\documentclass[11pt]{amsart}

\usepackage{hyperref}
\hypersetup{nesting=true,debug=true,naturalnames=true}
\usepackage{graphicx,amssymb,upref}
\usepackage[usenames,dvipsnames]{pstricks}
\usepackage{epsfig}
\usepackage{color}
\usepackage[latin1]{inputenc}
\usepackage[T1]{fontenc}
\usepackage{amsmath, amsthm}
\usepackage[english]{babel}
\usepackage{amssymb}
\usepackage{latexsym}
\usepackage{graphicx,amssymb,upref}
\usepackage[all]{xy}
\usepackage{tikz}
\usetikzlibrary{matrix}

\date{\today}

\newtheorem{theorem}{Theorem}[section]
\newtheorem{lemma}[theorem]{Lemma}
\newtheorem{proposition}[theorem]{Proposition}
\newtheorem{corollary}[theorem]{Corollary}
\theoremstyle{definition}
\newtheorem{definition}[theorem]{Definition}
\newtheorem{example}[theorem]{Example}

\newtheorem{apart}[theorem]{ }

\newcommand{\ot}{\otimes}
\newcommand{\co}{\circ}

\let\<\langle
\let\>\rangle

\let\uml\"

\title[Weak Hopf quasigroups and matched pairs of quasigroupoids]{Weak Hopf quasigroups and matched pairs of quasigroupoids}

\title{Weak Hopf quasigroups and matched pairs of quasigroupoids} 

\begin{document}
	
\maketitle
	
\begin{center}
{\bf Ram\'on Gonz\'{a}lez Rodr\'{\i}guez$^{a,b}$}.
\end{center}
	
\begin{center}
{\small \vspace{0.1cm}  [https://orcid.org/0000-0003-3061-6685]}
\end{center}
\begin{center}	{\small $^{a}$ CITMAga, 15782 Santiago de Compostela, Spain}
\end{center}
\begin{center}
{\small $^{b}$ Universidade de Vigo, Departamento de Matem\'{a}tica Aplicada II,  E-36310 Vigo, Spain\\email: rgon@dma.uvigo.es}
\end{center}
\vspace{0.1cm}
	
\begin{abstract}  
In this paper we introduce the notion of exact factorization of a quasigroupoid and the notion of  matched pair of quasigroupoids with common base. We prove that if $({\sf A}, {\sf H})$ is a matched pair of quasigroupoids it is posible to construct a new quasigroupoid  ${\sf A}\bowtie {\sf H}$ called the double cross product  of ${\sf A}$ and ${\sf H}$. Also, we show that, if a quasigroupoid ${\sf B}$ admits an exact factorization, there exists  a matched pair of quasigroupoids   $({\sf A}, {\sf H})$ and an isomorphism of quasigroupoids  between ${\sf A}\bowtie {\sf H}$ and ${\sf B}$. Finally, if ${\mathbb K}$ is a field, we show that every matched pair of quasigroupoids  $({\sf A}, {\sf H})$ induce, thanks to the quasigroupoid magma construction, a pair $({\mathbb K}[{\sf A}], {\mathbb K}[{\sf H}])$ of weak Hopf quasigroups and a double crossed product weak Hopf quasigroup ${\mathbb K}[{\sf A}]\bowtie{\mathbb K}[{\sf H}]$ isomorphic to  ${\mathbb K}[{\sf A}\bowtie {\sf H}]$ as weak Hopf quasigroups.
\end{abstract} 
	
\vspace{0.2cm} 
	
	
	
\section{Introduction} The notion of matched pair of groups was introduced by Takeuchi in \cite{TAK1} to formalize the situation that arises when we have two groups $(R, \bullet)$ and $(T, \star)$ acting on each other in a compatible way. The pair $(R,T)$ is a matched pair of groups if there exist two group actions $\varphi_{R}:T\times R\rightarrow R$ and $\phi_{T}:T\times R\rightarrow T$ such that 
$$\varphi_{R}(u, e_{R})=e_{R}, \;\; \varphi_{R}(u, r\bullet s)=\varphi_{R}(u,r)\bullet \varphi_{R}(\phi_{T}(u,r),s) $$
$$\phi_{T}(e_{T},r)=e_{T}, \;\; \phi_{T}(u\star v, r) =\phi_{T}(u,\varphi_{R}(v,r)) \star \phi_{T}(v,r).$$

Given a matched pair of groups $(R,T)$ it is possible to define a new group denoted by $R\bowtie T$ and called the double cross product group. This group is the set $R\times T$ with unit, product and inverse defined by $e_{R\bowtie T}=(e_{R}, e_{T})$, 
$$(r, u)\cdot (s,v)=(r\bullet \varphi_{R}(u,s), \phi_{T}(u,s)\star v)$$
and 
$$(r,u)^{-1}=(\varphi_{R}(u^{-1},s^{-1}), \phi_{T}(u^{-1},s^{-1})).$$
respectively. It is a relevant fact that this type of pairs solve the question of when a group $X$ admits an exact factorization $X=RT$, i.e., $R$ and $T$ are subgroups of $X$ and the map $R\times T\rightarrow X$, given by the product in $X$ is a bijection. In this setting $X=RT$ iff  $(R,T)$ is a matched pair of groups and $R\bowtie T$ is isomorphic to $X$ in the category of groups. 

Let ${\mathbb K}$ be a field, let ${\mathbb K}$-{\sf Vect} be the category of vector spaces over ${\mathbb K}$ and denote by $\otimes$ the tensor product in the category ${\mathbb K}$-{\sf Vect}. The double cross product of two Hopf algebras $A$ and $H$ in ${\mathbb K}$-{\sf Vect} was introduced by Majid in \cite[Proposition 3.12]{MAJ} (see also \cite[Theorem 7.2.2]{MAJDCP}) as a  Hopf algebra  defined in the tensor product $A\otimes H$ and determined by a matched pair of Hopf algebras $(A,H)$. A matched pair  of Hopf algebras is a pair $(A,H)$, where $A$ and $H$ are Hopf algebras, $A$ is a left $H$-module coalgebra with action $\varphi_{A}:H\otimes A\rightarrow A$, $H$ is a right $A$-module coalgebra with action $\phi_{H}:H\otimes A\rightarrow H$ and some suitable compatibility conditions hold for all $h,g\in H$ and $a,b\in A$. Using the Heyneman-Sweedler's convention these conditions can be written as follows:
$$\varphi_{A}(h\otimes 1_{A})=\varepsilon(h)1_{A}, \;\; \varphi_{A}( h\otimes ab)=\varphi_{A}(h_{(1)}\otimes a_{(1)})\varphi_{A}(\phi_{H}(h_{(2)} \otimes a_{(2)})\otimes b) $$
$$\phi_{H}(1_{H}\otimes a)=\varepsilon(a)1_{H}, \;\; \phi_{H}(hg\otimes a) =\phi_{H}(h \otimes  \varphi_{A}(g_{(1)} \otimes a_{(1)})) \phi_{H}(g_{(2)}\otimes a_{(2)}), $$
$$\phi_{H}(h_{(1)}\otimes a_{(1)})\otimes \varphi_{A}(h_{(2)}\otimes a_{(2)})=\phi_{H}(h_{(2)}\otimes a_{(2)})\otimes \varphi_{A}(h_{(1)}\otimes a_{(1)}).$$

If $(A,H)$ is a matched pair  of Hopf algebras, the double cross product $A\bowtie H$ of $A$ with $H$ is the Hopf algebra built on the vector space $A\otimes H$ with product 
$$(a\otimes h)(b\otimes g)=a\varphi_{A}(h_{(1)} \otimes b_{(1)})\otimes \phi_{H}(h_{(2)} \otimes b_{(2)})g$$
and tensor product unit, counit, coproduct and antipode 
$$\lambda_{A\bowtie H}(a\otimes h)=\lambda_{H}(h_{(2)})\triangleright \lambda_{A}(a_{(2)})\otimes \lambda_{H}(h_{(1)})\triangleleft \lambda_{A}(a_{(1)})$$
where $\lambda_{A}$ is the antipode of $A$ and $\lambda_{H}$ the antipode of $H$. Following \cite{MAJDCP} we can assure that this type of pairs solve the question of when a Hopf algebra admits an exact factorization as the product of two Hopf subalgebras. More concretely, a Hopf algebra $X$ factorises as $X=AH$ if there exists sub-Hopf algebras $A$ and $H$ with inclusion maps $i_{A}$ and $i_{H}$ such that the map $\omega (a\otimes h)=i_{A}(a)i_{H}(h)$ is an isomorphism of vector spaces. Then, as was proved by Majid in \cite[Theorem 7.2.3]{MAJDCP}, $X$ factorises as $X=AH$ iff there exists a matched pair of Hopf algebras $(A,H)$ such that $X$ is isomorphic to $A\bowtie H$ as Hopf algebras.

For any group $R$, its group algebra ${\mathbb K}[R]$ is an example of cocommutative Hopf algebra in the category ${\mathbb K}$-{\sf Vect} and, if $(R,T)$ is a matched pair of groups, the pair $({\mathbb K}[R], {\mathbb K}[T])$ is a matched pair of Hopf algebras. Moreover, there exists an isomorphism of Hopf algebras between ${\mathbb K}[R\bowtie T]$ and ${\mathbb K}[R]\bowtie {\mathbb K}[T]$. 

The theory of matched pairs of groups was extended by Mackenzie to groupoids in \cite{Mac}. In this context, it has been shown that there exist an equivalence between matched pairs of groupoids, exact factorizations and vacant double groupoids (see \cite{Mac}, \cite{AN}). Usign the groupoid algebra, the notion of  matched pair  can be extended to the theory of weak Hopf algebras and it is possible to obtain results that link exact factorizations of weak Hopf algebras with matched pairs and with double cross products of weak Hopf algebras (see \cite{GaPe}). For example,  if $({\sf A}, {\sf H})$ is a matched pair of finite groupoids with common base and, as in the previous cases $\bowtie$ denotes the double cross product, there exists an isomorphism of  weak Hopf algebras between ${\mathbb K}[{\sf A}\bowtie {\sf H}]$ and ${\mathbb K}[{\sf A}]\bowtie {\mathbb K}[{\sf H}]$ where ${\mathbb K}[{\sf A}]$, $ {\mathbb K}[{\sf H}]$ and ${\mathbb K}[{\sf A}]\bowtie {\mathbb K}[{\sf H}]$ are the corresponding groupoid algebras.

All the constructions mentioned in the previous paragraphs "live" in an associative world. Ask ourselves what happens when we replace  Hopf algebras, groupoids and weak Hopf algebras with their non-associative versions (Hopf quasigroups, quasigroupoids and weak Hopf quasigroups, respectively) is the main motivation of this paper. 

The notion of Hopf quasigroup in ${\mathbb K}$-{\sf Vect} was introduced  by Klim and Majid in \cite{KM} in order to understand the structure and relevant properties of the algebraic $7$-sphere, which is the loop of non-zero octonions. Hopf quasigroups are  particular cases  of unital coassociative $H$-bialgebras (see \cite{PI07}) and also of quantum quasigroups (see \cite{SM1} and \cite{SM2}). Hopf quasigroups include as an examples the quasigroup magma for a quasigroup in the sense of Klim and Majid, i.e., a loop  with the inverse property (see \cite{KM}, \cite{Bruck}), and also  the enveloping algebra $U(M)$ of a Malcev algebra living in a category of modules over a ring satisfying suitable conditions  (see \cite{KM} and \cite{PIS}). For these non-associative Hopf objects the theory of double cross products was developed in \cite{our2} and the corresponding one for quasigroups can be found in \cite{KM2}. It is important to highlight that in \cite{our2} the authors proved results similar to those existing for Hopf algebras relating two cocycles, skew pairings and double cross products. 

On the other hand,  quasigroupoids are the non-associative version of groupoids. This notion is equivalent to the one introduced by J. Grabowski in \cite{GRABO} with the name of inverse loopoid. As was pointed in \cite{GRABO22}, quasigroupoids are relevant in order to extend the Lie functor to categories of non-associative objects. From a purely algebraic point of view the category of finite quasigroupoids is equivalent to the one of pointed cosemisimple weak Hopf quasigroups over a given field ${\mathbb K}$. As a consequence, if ${\mathbb K}$ is algebraically closed, we obtain that  the categories of finite quasigroupoids and cocommutative cosemisimple weak Hopf quasigroups are equivalent (see \cite{JA21}). Weak Hopf quasigroups was introduced in \cite{Asian} by Alonso Álvarez, Fernández Vilaboa  and González Rodríguez as a new generalization of Hopf algebras  which encompass weak Hopf algebras and Hopf quasigroups.  A family of non-trivial examples of weak Hopf quasigroups can be obtained by working with  bigroupoids, i.e. bicategories where every $1$-cell is an equivalence and every $2$-cell is an isomorphism. We can also obtain interesting examples of this type of non-associative algebraic structures thanks to quasigroupoid magmas (see Example \ref{ex-k}). 

This paper introduces a theory of matched pairs and exact factorizations of quasigroupoids with common base such that, if $({\sf A}, {\sf H})$ is one of this matched pairs, the pair of its corresponding quasigroupoid magmas $({\mathbb K}[{\sf A}], {\mathbb K}[{\sf H}])$ has an associated isomorphism of weak Hopf quasigroups  between ${\mathbb K}[{\sf A}\bowtie {\sf H}]$ and a double crossed product ${\mathbb K}[{\sf A}]\bowtie{\mathbb K}[{\sf H}]$.  To obtain these results,  following the theory developed in the associative setting for groupoids, in the second section we define the notion of matched pair of quasigroupoids with common base and we prove that, if $({\sf A}, {\sf H})$ is a matched pair of quasigroupoids, it is posible to construct a new quasigroupoid denoted by ${\sf A}\bowtie {\sf H}$ and called the double crossed product or the diagonal quasigroupoid of ${\sf A}$ and ${\sf H}$. Also, in this section  the notion of exact factorization of a quasigroupoid ${\sf B}$ is introduced and it is proven that  ${\sf B}$ admits an exact factorization iff there exists  a matched pair of quasigroupoids   $({\sf A}, {\sf H})$ and an isomorphism of quasigroupoids  between ${\sf A}\bowtie {\sf H}$ and ${\sf B}$ induced by the product in ${\sf B}$.  In the third section it is proven that that every matched pair of quasigroupoids  $({\sf A}, {\sf H})$ induce a pair $({\mathbb K}[{\sf A}], {\mathbb K}[{\sf H}])$ of weak Hopf quasigroups and a  weak Hopf quasigroup ${\mathbb K}[{\sf A}]\bowtie{\mathbb K}[{\sf H}]$, obtained as a double cross product, such that ${\mathbb K}[{\sf A}\bowtie {\sf H}]$ and ${\mathbb K}[{\sf A}]\bowtie{\mathbb K}[{\sf H}]$ are isomorphic weak Hopf quasigroups. At this point it is necessary to point out that there is not yet a solid theory of double cross products for weak Hopf quasigroups. In any case,  if it exists, it should contain as a particular case the construction of ${\mathbb K}[{\sf A}]\bowtie{\mathbb K}[{\sf H}]$ introduced in this paper. As a particular case of this general theory we can prove  that, if $(A, H)$ is a matched pair of quasigroups, ${\mathbb K}[{A}\bowtie {H}]$ and ${\mathbb K}[{A}]\bowtie {\mathbb K}[{ H}]$ are isomorphic as Hopf quasigroups.

\section{Matched pairs of quasigroupoids}

A quasigroupoid is the weak version of the notion of quasigroup (or IP loop)  given by J. Klim and S. Majid  in \cite{KM}. The relationship between quasigroupoids and quasigroups is similar to that between groupoids and groups. In fact, we can also think of a quasigroupoid as the non-associative version of the notion of groupoid.  We will start this section by recalling the notions of quasigroup introduced in \cite{KM} and the one of quasigroupoid defined in \cite{JA21}.

\begin{definition}
\label{IP-loop}
{\rm 
A quasigroup $A$ is a triple $A=(A, \cdot, e_{A})$, where  $A$ is a set equipped with a  product $\cdot$,  an  element $e_A$, called identity element,  satisfying 
$$
 e_A\cdot u=u=u\cdot e_A
 $$
and with the property that for each $u\in A$ there exists $u^{-1}\in A$ such that 
$$
u^{-1}\cdot(u\cdot v)=v=(v\cdot u)\cdot u^{-1},
$$
holds for all $v\in A$. 

It is easy to prove that, in any quasigroup  $L$, the inverse of an element $u$ is unique and
$$
(u^{-1})^{-1}=u,\;\;\;(u\cdot v)^{-1}=v^{-1}\cdot u^{-1},  $$
hold for all $u,v\in A$. Obviously, a group is a quasigroup where the product is associative.
}
\end{definition} 

\begin{definition}
\label{XY}
{\rm
Let $X$, $Y$ and $P$ be  sets and let $i:X\rightarrow P$ and $j:Y\rightarrow P$ be maps. We define the set $X\;_{i}\hspace{-0.1cm}\times_{j}Y$  by 
$$X\;_{i}\hspace{-0.1cm}\times_{j}Y=\{(x,y)\in X\times Y \; / \;i(x)=j(y)\}.$$

Similarly, if $Z$ is a set and  $k:Z\rightarrow P$, $l:Y\rightarrow P$  maps, we define $X\;_{i}\hspace{-0.1cm}\times_{j}Y\;_{l}\hspace{-0.1cm}\times_{k}Z$ by 
$$X\;_{i}\hspace{-0.1cm}\times_{j}Y\;_{l}\hspace{-0.1cm}\times_{k}Z=\{(x,y,z)\in X\times Y\times  Z\; / \;i(x)=j(y), \; l(y)=k(z) \}.$$

}
\end{definition}

\begin{definition}
\label{quasigroupoid}
{\rm 
A quasigroupoid ${\sf A}$ is an ordered pair of sets ${\sf A}=({\sf A}_0,{\sf A}_1)$ such that:

\begin{itemize}
\item[(a1)] There exist maps $s_{\sf A}:{\sf A}_1\rightarrow {\sf A}_0,$ $t_{\sf A}:{\sf A}_1\rightarrow {\sf A}_0,$ and $id_{\sf A}:{\sf A}_0\rightarrow {\sf A}_1,$ called source, target  and identity, respectively, satisfying 
$$s_{\sf A}(id_{\sf A}(x))=t_{\sf A}(id_{\sf A}(x))=x,\;\;\forall \; x\in {\sf A}_0.$$  
\item[(a2)] There exist a map, called product of ${\sf A}$,  
$$\bullet :{\sf A}_1\;_{s_{\sf A}}\hspace{-0.15cm}\times_{t_{\sf A}} {\sf A}_1\rightarrow {\sf A}_1,$$ defined by $\bullet(a,b)=a\bullet b$ and a map
$\lambda_{\sf A}:{\sf A}_1\rightarrow {\sf A}_1$, called the inverse map,  such that:

\begin{itemize}
\item[(${\rm a2-1}$)] For each $a\in {\sf A}_1$,  $$id_{{\sf A}}(t_{{\sf A}}(a))\bullet a=a=a\bullet  id_{{\sf A}}(s_{{\sf A}}(a)).$$
\item[(${\rm a2-2}$)] For all $(a, b)\in {\sf A}_1\;_{s_{\sf A}}\hspace{-0.15cm}\times_{t_{\sf A}} {\sf A}_1$, $$s_{{\sf A}}(a\bullet b)=s_{{\sf A}}(b), \;\;t_{{\sf A}}(a\bullet b)=t_{{\sf A}}(a).$$ 
\item[(${\rm a2-3}$)] For all $(a, b)\in {\sf A}_1\;_{s_{\sf A}}\hspace{-0.15cm}\times_{t_{\sf A}} {\sf A}_1$, $(\lambda_{\sf A}(a), a\bullet b)$ and $(a\bullet b, \lambda_{\sf A}(b))$ are in ${\sf A}_1\;_{s_{\sf A}}\hspace{-0.15cm}\times_{t_{\sf A}} {\sf A}_1$ and 
$$\lambda_{\sf A}(a)\bullet (a\bullet b)=b,\;\;\; (a\bullet b)\bullet \lambda_{{\sf A}}(b)=a.$$ 
\end{itemize}
\end{itemize}

The set ${\sf A}_{0}$ wil be called the base of ${\sf A}$. We will say that a quasigroupoid ${\sf A}$ is finite if its base is a finite set. Note that a finite quasigroupoid where $\vert {\sf A}_{0}\vert =1$ is a quasigroup.
}
\end{definition}

As was pointed in \cite{JA21} a quasigroupoid in the sense of Definition \ref{quasigroupoid} is an inverse semiloopoid satisfying  the unities associativity assumption, i.e., an inverse loopoid in the sense of \cite{GRABO} (see \cite[Definition 5.2]{GRABO}. As a consequence, we have that the following equalities
\begin{align}
\label{E-1}
s_{\sf A}(\lambda_{\sf A}(a))&=t_{\sf A}(a),\\
\label{E-2}
t_{\sf A}(\lambda_{\sf A}(a))&=s_{\sf A}(a),\\ 
\label{E-3}
\lambda_{\sf A}(a)\bullet a &=id_{\sf A}(s_{\sf A}(a)),\\
\label{E-4}
a\bullet \lambda_{\sf A}(a)&=id_{\sf A}(t_{\sf A}(a)),\\
\label{E-5}
\lambda_{\sf A}(\lambda_{\sf A}(a))&=a,\\
\label{E-6}
\lambda_{\sf A}(a\bullet b)&=\lambda_{\sf A}(b)\bullet \lambda_{\sf A}(a),
\end{align}
hold for all $a\in {\sf A}_1$ and $(a, b)\in {\sf A}_1\;_{s_{\sf A}}\hspace{-0.15cm}\times_{t_{\sf A}} {\sf A}_1$.

\begin{definition}
\label{mor-loop}
{\rm 
Let ${\sf A}$, ${\sf A}^{\prime}$ be quasigroupoids. A morphism $\Gamma:{\sf A}\rightarrow {\sf A}^{\prime}$ between ${\sf A}$ and ${\sf A}^{\prime}$ is a pair of maps $\Gamma=(\Gamma_{0},\Gamma_{1}),$  $\Gamma_{0}:{\sf A}_{0}\rightarrow {\sf A}^{\prime}_{0},$ $\Gamma_{1}:{\sf A}_{1}\rightarrow {\sf A}^{\prime}_{1}$,  such that 
\begin{itemize}
\item[(b1)] $\Gamma_{0}\circ s_{\sf A} =s_{{\sf A}^{\prime}}\circ \Gamma_{1},$
\item[(b2)] $\Gamma_{0}\circ t_{\sf A}  = t_{{\sf A}^{\prime}}\circ  \Gamma_{1},$ 
\item[(b3)] $\Gamma_{1}(id_{\sf A}(x)) = id_{{\sf A}^{\prime}}(\Gamma_{0}(x)),$
\end{itemize}
hold for all $x\in {\sf A}_{0}$, and 
\begin{itemize}
\item[(b4)]  $\Gamma_{1}(a\bullet b)=\Gamma_{1}(a)\bullet^{\prime}\Gamma_{1}(b),$
\end{itemize}
holds for all $(a,b)\in {\sf A}_1\;_{s_{\sf A}}\hspace{-0.15cm}\times_{t_{\sf A}} {\sf A}_1$.

The obvious composition of quasigroupoid morphisms is a quasigroupoid morphism. Then with {\sf QGPD} we will denote the category whose objects are  quasigroupoids and whose morphisms are  morphisms of  quasigroupoids.  Note that $\Gamma$ is a monomorphism in ${\sf QGPD}$ if the maps $\Gamma_{0}$ and $\Gamma_{1}$ are injective, $\Gamma$ is an epimorphism in ${\sf QGPD}$ if the maps $\Gamma_{0}$ and $\Gamma_{1}$ are sobrejective and $\Gamma$ is an isomorphism in ${\sf QGPD}$ if the maps $\Gamma_{0}$ and $\Gamma_{1}$ are bijective.
}
\end{definition}

\begin{example}
\label{exquasi}
{\rm  The following example of quasigrouipoid was introduced in \cite[Example 2.6]{JA21}.  Let $A$ be a quasigroup with product $\cdot$ and let $X$ be a set. Assume that there exists a map  $\psi_{X}:A\times X\rightarrow X$ satisfying the following two conditions:
$$
\psi_{X}(e_{A},x)=x,\;\;\;\;\;\;
\psi_{X}(a\cdot b,x)=\psi_{X}(a, \psi_{X}(b, x)), 
$$
for all $x\in X$ and $a,b\in A.$
	
In this case we will say that $\psi_{X}$ is an action of $A$ over $X$. The  quasigroupoid ${\sf B}=({\sf B}_0,{\sf B}_1)$ associated to the action $\psi_{X}$ is defined by the sets 
${\sf B}_0=X,$ ${\sf B}_1=A\times X$ and maps
$$s_{{\sf B}}:{\sf B}_1\rightarrow {\sf B}_0, \;\;\;  s_{{\sf B}}(a,x)=x,$$
$$t_{{\sf B}}:{\sf B}_1\rightarrow {\sf B}_0,\;\;\; t_{{\sf B}}(a,x)=\psi_{X}(a, x), $$
$$id_{{\sf B}}:{\sf B}_0\rightarrow {\sf B}_1, \;\;\;  id_{{\sf B}}(x)=(e_{A},x).$$
	
Then, ${\sf B}_1\; _{s_{\sf B}}\hspace{-0.15cm}\times_{t_{\sf B}} {\sf B}_1=\{((a,x), (b,y),)\in {\sf B}_1\times {\sf B}_1 \;/ \; \psi_{X}(b, y)=x\}$ and the product  is defined by  $(a,x)\star (b,y)= ( a\cdot b,y).$   The inverse map $\lambda_{\sf B}:\sf{B}_1\rightarrow \sf{B}_1$ is 
$ \lambda_{\sf{B}}(a,x)=(a^{-1}, \psi_{X}(a, x)).$ 

As was proved in \cite[Example 2.6]{JA21} examples of this kind can be obtained by working with Moufang loops of small order (Moufang loops fit in Definition \ref{IP-loop}) as the ones introduced by O. Chein in \cite{CHEIN} and the 4-dimensional Taft Hopf algebra.

}
\end{example}

\begin{example}
\label{GB1}
{\rm The following example is the algebraic version of \cite[Example 3.8]{GRABO22}. Let $A$ be a quasigroup and let $X$ be a set. Denote by $T$ the set 
$$T=\{(a,x,y)\;/\; a\in A, \;x,y\in X\}.$$

The  quasigroupoid ${\sf T}=({\sf T}_0,{\sf T}_1)$ associated to $T$ is defined by the sets 
${\sf T}_0=\{(e_{A},x,x)\in T\},$ ${\sf T}_1=T$ and maps
$$s_{{\sf T}}:{\sf T}_1\rightarrow {\sf T}_0, \;\;\;  s_{{\sf T}}(a,x,y)=(e_{A},y,y),$$
$$t_{{\sf T}}:{\sf T}_1\rightarrow {\sf T}_0,\;\;\; t_{{\sf T}}(a,x,y)=(e_{A},x,x), $$
$$id_{{\sf T}}:{\sf T}_0\rightarrow {\sf T}_1, \;\;\;  id_{{\sf T}}(e_{A},x,x)=(e_{A},x,x).$$

Then, $${\sf T}_1\; _{s_{\sf T}}\hspace{-0.15cm}\times_{t_{\sf T}} {\sf T}_1=\{((a,x,y), (b,y,r))\in {\sf T}_1\times {\sf T}_1 \}$$ 
and the product  is defined by  $$(a,x,y)\star (b,y,r)= ( a\cdot b,x,r).$$  

The inverse map $\lambda_{\sf T}:\sf{T}_1\rightarrow \sf{T}_1$ is 
$$ \lambda_{\sf{B}}(a,x,y)=(a^{-1}, y,x).$$
}
\end{example}

\begin{example}
\label{GB2}
{\rm The last example is the algebraic version of  \cite[Example 3.9]{GRABO22}. In this case, let ${\sf A}$ be a quasigroupoid, let $P$ be a set and let $\pi:P\rightarrow {\sf A}_{0}$ be a surjective map. The  quasigroupoid ${\sf P(A)^{\pi}=(P(A)^{\pi}_0,P(A)^{\pi}_1)}$  is defined by the sets 
${\sf P(A)}^{\pi}_0=P,$ 
$${\sf P(A)}^{\pi}_1= \{(p,a,q)\;/\; (p,a,q)\in P\times {\sf A}_{1}\times P, \; \pi(p)=t_{\sf A}(a), \; \pi(q)=s_{\sf A}(a)\}$$
and maps
$$s_{{\sf P(A)}^{\pi}}:{\sf P(A)}^{\pi}_1\rightarrow {\sf P(A)}^{\pi}_0, \;\;\;  s_{{\sf P(A)}^{\pi}}(p,a,q)=q,$$
$$t_{{\sf P(A)}^{\pi}}:{\sf P(A)}^{\pi}_1\rightarrow {\sf P(A)}^{\pi}_0,\;\;\; t_{{\sf P(A)}^{\pi}}(p,a,q)=p, $$
$$id_{{\sf P(A)}^{\pi}}:{\sf P(A)}^{\pi}_0\rightarrow {\sf P(A)}^{\pi}_1, \;\;\;  id_{{\sf P(A)}^{\pi}}(p,a,q)=(p,id_{\pi(p)},p).$$
	
Then, $${\sf P(A)}^{\pi}_1\; _{s_{{\sf P(A)}^{\pi}}}\hspace{-0.15cm}\times_{t_{{\sf P(A)}^{\pi}}} {{\sf P(A)}^{\pi}}_1=\{((p,a,q), (q,b,n))\in {{\sf P(A)}^{\pi}}_1\times {{\sf P(A)}^{\pi}}_1 \}$$ 
and the product  is defined by  $$(p,a,q)\star (q,b,n)= ( p, a\cdot b, n).$$  
	
The inverse map $\lambda_{{\sf P(A)}^{\pi}}:{\sf P(A)}^{\pi}_1\rightarrow {\sf P(A)}^{\pi}_1$ is 
$$ \lambda_{{\sf P(A)}^{\pi}}(p,a,q)=(q, \lambda_{\sf A}(a), p).$$
}
\end{example}

\begin{apart}
\label{groupoid}
{\rm 
As was pointed in the beginning of this section, a quasigroupoid is the non-associative version of the notion of groupoid.   The notion of groupoid was introduced, using sets and unary and binary operations, by H. Brandt  in a 1926 paper on the composition of quadratic forms in four variables \cite{BRANDT} (see also \cite{HAHN}).  Recall that a groupoid  ${\sf G}$ is an ordered pair of sets ${\sf G}=(\sf{G}_0,\sf{G}_1)$ such that:
\begin{itemize}
\item[(a$^{\prime}$1)] There exist maps $s_{\sf G}:\sf{G}_1\rightarrow \sf{G}_0,$ $t_{\sf G}:\sf{G}_1\rightarrow \sf{G}_0,$ and $id_{\sf G}:\sf{G}_0\rightarrow \sf{G}_1,$ called source, target and identity, respectively, satisfying 
$$s_{\sf G}(id_{\sf G}(x))=t_{\sf G}(id_{\sf G}(x))=x,\;\;\; \forall\;  x\in \sf{G}_0.$$  
\item[($a^{\prime}$2)] There exist a map, called product of $\sf{G}$,  
$$\bullet:\sf{G}_1\;_{s_{\sf G}}\hspace{-0.15cm}\times_{t_{\sf G}}  \sf{G}_1\rightarrow \sf{G}_1,$$ defined by $\bullet(\tau,\sigma)=\tau\bullet \sigma,$ and a map
$\lambda_{\sf G}:\sf{G}_1\rightarrow \sf{G}_1$, called the inverse map, such that: 
			
\begin{itemize}
\item[(${\rm a}^{\prime}2-1$)] For each $\sigma\in \sf{G}_1$, $$id_{\sf G}(t_{\sf G}(\sigma))\bullet \sigma=\sigma=\sigma\bullet  id_{\sf G}(s_{\sf G}(\sigma)).$$
\item[(${\rm a}^{\prime}2-2$)] For all $(\tau, \sigma)\in \sf{G}_1\;_{s_{\sf G}}\hspace{-0.15cm}\times_{t_{\sf G}}  \sf{G}_1$, $$s_{\sf G}(\tau\bullet \sigma)=s_{\sf G}(\sigma), \;\;t_{\sf G}(\tau\bullet \sigma)=t_{\sf G}(\tau).$$ 
\item[(${\rm a}^{\prime}2-3$)] If either $\omega\bullet (\tau\bullet \sigma)$ or $(\omega\bullet \tau)\bullet \sigma$  is defined so is the other and $\omega\bullet (\tau\bullet \sigma)=(\omega\bullet \tau)\bullet \sigma$.      
\item[(${\rm a}^{\prime}2-4$)] For each $\sigma\in \sf{G}_1$, $(\lambda_{\sf G}(\sigma),\sigma)$ and  $(\sigma, \lambda_{\sf G}(\sigma))$ are in $\sf{G}_1\;_{s_{\sf G}}\hspace{-0.15cm}\times_{t_{\sf G}}  \sf{G}_1$ and 
$$\lambda_{\sf G}(\sigma)\bullet \sigma=id_{\sf G}(s_{\sf G}(\sigma)), \;\;\;\; \sigma\bullet \lambda_{\sf G}(\sigma)=id_{\sf G}(t_{\sf G}(\sigma)).$$
\end{itemize}
\end{itemize}

From another point of view  a groupoid $\sf{G}$ is a small category in which every morphism is an isomorphism.  A group is a particular example of a groupoid, with only one object, where  the elements of the group correspond to morphisms of the groupoid. Also, a groupoid morphism is a functor between groupoids or, in terms of sets and operations, a pair of maps satisfying (b1)-(b4) of Definition \ref{mor-loop}. Then we have a category of groupoids, denoted by ${\sf GPD}$, and it is obvious that it is a subcategory of ${\sf QGPD}$.

A typical example of groupoid is the coarse groupoid ${\sf X}^{c}$ associated to a set $X$. This groupoid is defined as follows: 
$${\sf X}^{c}_{0}= X, \;\;  {\sf X}^{c}_{1}= X\times X, \;\; s_{{\sf X}^{c}}(x,y)=y, \;\; t_{{\sf X}^{c}}(x,y)=x, \;\; id_{{\sf X}^{c}}(x)=(x,x)$$
$$ (z,x)\bullet (x,y)=(z,y), \;\; \lambda_{{\sf X}^{c}}(x,y)=(y,x).$$

The discrete groupoid ${\sf X}^{d}$ associated to $X$ is other  example of groupoid where ${\sf X}^{d}_{0}= X, \;  {\sf X}^{d}_{1}= X, \; s_{{\sf X}^{d}}(x)=t_{{\sf X}^{d}}(x)=id_{{\sf X}^{d}}(x)=\lambda_{{\sf X}^{d}}(x)=x$ and $x\bullet x=x$.

}
\end{apart}

\begin{definition}
\label{sub}
{\rm Let ${\sf A}$ and ${\sf B}$ be quasigroupoids. We will say that ${\sf A}$ is a subquasigroupoid of ${\sf B}$ if there exists a monomorphism  $i^{{\sf A}}:{\sf A}\rightarrow {\sf B}$ in  {\sf QGPD}.
}
\end{definition}

\begin{definition}
\label{action}
{\rm Let ${\sf A}$ and ${\sf H}$ be quasigroupoids with products $\bullet$ and $\star$  and with the same base.  A left action of ${\sf H}$ on ${\sf A}$ is a map $\varphi_{A} :{\sf H}_1\;_{s_{\sf H}}\hspace{-0.15cm}\times_{t_{\sf A}} {\sf A}_1\rightarrow {\sf A}_1,$  satisfying:
\begin{itemize}
\item[(c1)] For all $(h,a)\in {\sf H}_1\;_{s_{\sf H}}\hspace{-0.15cm}\times_{t_{\sf A}} {\sf A}_1$, $$t_{\sf A}(\varphi_{\sf A}(h,a)) =t_{{\sf H}}(h).$$
\item[(c2)] For all $(h,a)\in {\sf H}_1\;_{s_{\sf H}}\hspace{-0.15cm}\times_{t_{\sf A}} {\sf A}_1$ and  $(g, h)\in {\sf H}_1\;_{s_{\sf H}}\hspace{-0.15cm}\times_{t_{\sf H}} {\sf H}_1$,   $$\varphi_{\sf A}(g\star h, a)=\varphi_{\sf A}(g, \varphi_{\sf A}(h, a)).$$
\item[(c3)]  For all $a\in {\sf A}_1$, 
$$\varphi_{\sf A}(id_{\sf H}(t_{\sf A}(a)), a)=a.$$
\end{itemize}

Similarly, a right action  of ${\sf A}$ on ${\sf H}$ is a map $\phi_{\sf H} :{\sf H}_1\;_{s_{\sf H}}\hspace{-0.15cm}\times_{t_{\sf A}} {\sf A}_1\rightarrow {\sf H}_1,$  satisfying:
\begin{itemize}
	\item[(d1)] For all $(h,a)\in {\sf H}_1\;_{s_{\sf H}}\hspace{-0.15cm}\times_{t_{\sf A}} {\sf A}_1$, $$s_{\sf H}(\phi_{\sf H}(h,a)) =s_{{\sf A}}(a).$$
	\item[(d2)] For all $(h,a)\in {\sf H}_1\;_{s_{\sf H}}\hspace{-0.15cm}\times_{t_{\sf A}} {\sf A}_1$ and  $(a, b)\in {\sf A}_1\;_{s_{\sf A}}\hspace{-0.15cm}\times_{t_{\sf A}} {\sf A}_1$,   $$\phi_{\sf H}( h, a\bullet b)=\phi_{\sf H}(\phi_{\sf H}(h, a), b),$$
	\item[(d3)]  For all $h\in {\sf H}_1$, 
	$$\phi_{\sf H}(h, id_{\sf A}(s_{\sf H}(h)))=h.$$
\end{itemize}
}
\end{definition}

\begin{definition}
\label{mp}
{\rm A matched pair of quasigroupoids is a pair of quasigroupoids $({\sf A}, {\sf H})$ with the same base together with a left action of ${\sf H}$ on ${\sf A}$, $\varphi_{\sf A} :{\sf H}_1\;_{s_{\sf H}}\hspace{-0.15cm}\times_{t_{\sf A}} {\sf A}_1\rightarrow {\sf A}_1,$ and a  right action  of ${\sf A}$ on ${\sf H}$, $\phi_{\sf H} :{\sf H}_1\;_{s_{\sf H}}\hspace{-0.15cm}\times_{t_{\sf A}} {\sf A}_1\rightarrow {\sf H}_1$,  satisfying  the following properties: 
\begin{itemize}
	\item[(e1)] For all $(h,a)\in {\sf H}_1\;_{s_{\sf H}}\hspace{-0.15cm}\times_{t_{\sf A}} {\sf A}_1$, $$s_{\sf A}(\varphi_{\sf A}(h,a)) =t_{{\sf H}}(\phi_{\sf H}(h,a)).$$
	\item[(e2)] For all $(h,a)\in {\sf H}_1\;_{s_{\sf H}}\hspace{-0.15cm}\times_{t_{\sf A}} {\sf A}_1$ and  $(a, b)\in {\sf A}_1\;_{s_{\sf A}}\hspace{-0.15cm}\times_{t_{\sf A}} {\sf A}_1$,   $$\varphi_{\sf A}( h, a\bullet b)=\varphi_{\sf A}(h, a)\bullet \varphi_{\sf A}(\phi_{\sf H}(h,a), b) .$$
	\item[(e3)]  For all $(h,a)\in {\sf H}_1\;_{s_{\sf H}}\hspace{-0.15cm}\times_{t_{\sf A}} {\sf A}_1$ and  $(g, h)\in {\sf H}_1\;_{s_{\sf H}}\hspace{-0.15cm}\times_{t_{\sf H}} {\sf H}_1$,  $$\phi_{\sf H}(g\star h, a)=\phi_{\sf H}(g, \varphi_{\sf A}(h, a))\star \phi_{\sf H}(h, a).$$
\end{itemize}

The notion of  matched pair of quasigroupoids is similar to the notion of matched pair of groupoids (see \cite[Definition 1.1]{Mar}).  It is interesting to emphasize that although we are working in a non-associative context, all the fundamental properties of this type of pairs collected in 
\cite[Lemma  1.2]{Mar} can be proved.
}
\end{definition}

\begin{proposition}
\label{properties}
Let  $({\sf A}, {\sf H})$ be a matched pair of quasigroupoids. For all $a, b\in {\sf A}_1$, $g, h\in {\sf H}_1$ for which the operations are defined, the following identities hold: 
\begin{align}
\label{P-1}
\varphi_{\sf A}(h,id_{\sf A}(s_{\sf H}(h) ))& =id_{\sf A}(t_{\sf H}(h)),\\
\label{P-2}
\phi_{\sf H}(id_{\sf H}(t_{\sf A}(a)), a)& =id_{\sf H}(s_{\sf A}(a)),\\ 
\label{P-3}
\lambda_{\sf A}(\varphi_{\sf A}(h,a))& = \varphi_{\sf A}(\phi_{\sf H}(h,a), \lambda_{\sf A}(a)),\\
\label{P-4}
\lambda_{\sf H}(\phi_{\sf H}(h,a))& = \phi_{\sf H}( \lambda_{\sf H}(h), \varphi_{\sf A}(h,a))\\
\label{P-5}
(b\bullet \varphi_{\sf A}(h,a))\bullet \varphi_{\sf A}(\phi_{\sf H}(h,a), \lambda_{\sf A}(a))&= b,\\
\label{P-6}
\phi_{\sf H}( \lambda_{\sf H}(h), \varphi_{\sf A}(h,a)) \star (\phi_{\sf H}(h,a)\star g) &= g,\\
\label{P-7}
\varphi_{\sf A}( \lambda_{\sf H}(\phi_{\sf H}(h,a)), \lambda_{\sf A}( \varphi_{\sf A}(h,a))) & = \lambda_{\sf A}(a),\\
\label{P-8}
\phi_{\sf H}( \lambda_{\sf H}(\phi_{\sf H}(h,a)), \lambda_{\sf A}( \varphi_{\sf A}(h,a))) & = \lambda_{\sf H}(h),\\
\label{P-9}
\lambda_{\sf A}(a)\bullet  \varphi_{\sf A}(\lambda_{\sf H}(h),b) &=  \varphi_{\sf A}(\lambda_{\sf H} (\phi_{\sf H}(h,a)),  \lambda_{\sf A}( \varphi_{\sf A}(h,a))\bullet b),\\
\label{P-10}
\phi_{\sf H}(g,\lambda_{\sf A}(a))\star \lambda_{\sf H}(h)  &=\phi_{\sf H}(g\star \lambda_{\sf H} (\phi_{\sf H}(h,a)), \lambda_{\sf A}(\varphi_{\sf A}(h,a))).
\end{align}
\end{proposition}

\begin{proof} For any $h\in  {\sf H}_1$, by (e2) of Definition \ref{mp},  we have that 
\begin{align*}
\varphi_{\sf A}(h,id_{\sf A}(s_{\sf H}(h) ))& = \varphi_{\sf A}(h,id_{\sf A}(s_{\sf H}(h)) \bullet id_{\sf A}(s_{\sf H}(h))) \\
	\;	& =\varphi_{\sf A}(h,id_{\sf A}(s_{\sf H}(h) ))\bullet \varphi_{\sf A}( \phi_{\sf H}(h,id_{\sf A}(s_{\sf H}(h) )),id_{\sf A}(s_{\sf H}(h) )).
\end{align*}

Then, by (d3) of Definition \ref{action}, $\varphi_{\sf A}(h,id_{\sf A}(s_{\sf H}(h) ))=\varphi_{\sf A}(h,id_{\sf A}(s_{\sf H}(h) ))\bullet \varphi_{\sf A}(h,id_{\sf A}(s_{\sf H}(h) ))$ holds 
and, using  (${\rm a2-3}$) of Definition \ref{quasigroupoid}, (\ref{E-4}) and  (c1) of Definition \ref{action}, we can assure that 
$$\varphi_{\sf A}(h,id_{\sf A}(s_{\sf H}(h) ))=id_{\sf A}(t_{\sf A}(\varphi_{\sf A}(h,id_{\sf A}(s_{\sf H}(h)))))=id_{\sf A}(t_{\sf H}(h)).$$

Therefore, (\ref{P-1}) holds. The proof for (\ref{P-3}) is the following: By (e2) of Definition \ref{mp}, (\ref{E-4}), (\ref{P-1})  and (c1) of Definition \ref{action} we have that
\begin{align*}
\varphi_{\sf A}(h,a)\bullet  \varphi_{\sf A}(\phi_{\sf H}(h,a), \lambda_{\sf A}(a))& = \varphi_{\sf A}(h,a\bullet \lambda_{\sf A}(a))\\
\;	& =\varphi_{\sf A}(h, id_{\sf A}(t_{\sf A}(a)))\\
\;	& =\varphi_{\sf A}(h, id_{\sf A}(s_{\sf H}(h)))\\
\;	& =id_{\sf A}(t_{\sf H}(h))\\
\;	& = id_{\sf A}(t_{\sf A}(\varphi_{\sf A}(h,a)))
\end{align*}
holds. Then, by (${\rm a2-3}$) of Definition \ref{quasigroupoid}, we obtain (\ref{P-3}). 

The identity (\ref{P-5}) follows by (\ref{P-3}) and (${\rm a2-3}$) of Definition \ref{quasigroupoid}. On the other hand, by (\ref{P-3}), (c2) of Definition \ref{action},  (\ref{E-3}),  (d1) of Definition \ref{action} and  (\ref{P-2}), 
\begin{align*}
\varphi_{\sf A}( \lambda_{\sf H}(\phi_{\sf H}(h,a)), \lambda_{\sf A}( \varphi_{\sf A}(h,a)))& = \varphi_{\sf A}( \lambda_{\sf H}(\phi_{\sf H}(h,a)),  \varphi_{\sf A}(\phi_{\sf H}(h,a), \lambda_{\sf A}(a))) \\
\;	& =\varphi_{\sf A}( \lambda_{\sf H}(\phi_{\sf H}(h,a))\star \phi_{\sf H}(h,a), \lambda_{\sf A}(a))\\
\;	& =\varphi_{\sf A}( id_{\sf H}(s_{\sf A}(a)), \lambda_{\sf A}(a))\\
\;	& =\varphi_{\sf A}( id_{\sf H}(t_{\sf A}(\lambda_{\sf A}(a))), \lambda_{\sf A}(a))\\
\;	& = \lambda_{\sf A}(a)
\end{align*}
which proves (\ref{P-7}) .

The proofs of  (\ref{P-2}),  (\ref{P-4}), (\ref{P-6}) and (\ref{P-8}) are similar. Furthermore, by (e2) of Definition \ref{mp}, (\ref{P-7}) and (\ref{P-8}), 
$$\varphi_{\sf A}(\lambda_{\sf H} (\phi_{\sf H}(h,a)),  \lambda_{\sf A}( \varphi_{\sf A}(h,a))\bullet b)$$
$$= 
\varphi_{\sf A}(\lambda_{\sf H} (\phi_{\sf H}(h,a)),  \lambda_{\sf A}( \varphi_{\sf A}(h,a)))\bullet 
\varphi_{\sf A}(\phi_{\sf H}(\lambda_{\sf H} (\phi_{\sf H}(h,a)), \lambda_{\sf A} (\varphi_{\sf A}(h,a)), b)=
\lambda_{\sf A}(a)\bullet  \varphi_{\sf A}(\lambda_{\sf H}(h),b)$$
and then (\ref{P-9}) holds. The proof of (\ref{P-10}) is similar and we left the details to the reader.
\end{proof}

\begin{theorem}
\label{match-1}
Let  $({\sf A}, {\sf H})$ be a matched pair of quasigroupoids. The pair 
$${\sf A}\bowtie {\sf H}=(({\sf A}\bowtie {\sf H})_{0}, ({\sf A}\bowtie {\sf H})_{1}),$$ where $({\sf A}\bowtie {\sf H})_{0}={\sf A}_{0}$, $({\sf A}\bowtie {\sf H})_{1}= {\sf A}_1\;_{s_{\sf A}}\hspace{-0.15cm}\times_{t_{\sf H}} {\sf H}_1$, is a quasigroupoid with source morphism $s_{{\sf A}\bowtie {\sf H}}(a,h)=s_{\sf H}(h)$, target morphism $t_{{\sf A}\bowtie {\sf H}}(a,h)=t_{\sf A}(a)$, identity map $id_{{\sf A}\bowtie {\sf H}}(x)=
(id_{\sf A}(x), id_{\sf H}(x))$, product 
$$ (a,g)._{\Psi}(b,h)=(a\bullet \Psi_{1}(g,b),\Psi_{2}(g,b)\star h ),$$
where $\Psi: {\sf H}_1\;_{s_{\sf H}}\hspace{-0.15cm}\times_{t_{\sf A}} {\sf A}_1\rightarrow {\sf A}_1\;_{s_{\sf A}}\hspace{-0.15cm}\times_{t_{\sf H}} {\sf H}_1$ is the map with components $\Psi_{1}(g,b)=\varphi_{\sf A}(g,b)$ and  $\Psi_{2}(g,b)=\phi_{\sf H}(g,b)$, and inverse map 
$$\lambda_{{\sf A}\bowtie {\sf H}}(a,h)=\Psi (\lambda_{\sf H}(h), \lambda_{\sf A}(a)).$$
\end{theorem}

\begin{proof} First, note by (a1) of Definition \ref{quasigroupoid}, for all $x\in ({\sf A}\bowtie {\sf H})_{0}$ we have 
$$s_{{\sf A}\bowtie {\sf H}}(id_{{\sf A}\bowtie {\sf H}}(x))=s_{{\sf A}\bowtie {\sf H}}(id_{\sf A}(x), id_{\sf H}(x))=s_{\sf H}(id_{\sf H}(x))=x$$
and similarly $t_{{\sf A}\bowtie {\sf H}}(id_{{\sf A}\bowtie {\sf H}}(x))=x.$
Therefore, (a1) of Definition \ref{quasigroupoid} holds for ${\sf A}\bowtie {\sf H}$.
	
Secondly,  the product $._{\Psi}$ is well defined because, if $((a,g), (b,h))\in ({\sf A}\bowtie {\sf H})_{1}\;_{s_{{\sf A}\bowtie {\sf H}}}\hspace{-0.15cm}\times_{t_{{\sf A}\bowtie {\sf H}}} ({\sf A}\bowtie {\sf H})_{1}
$, we have that 
$$s_{\sf H}(g)=s_{{\sf A}\bowtie {\sf H}}(g,b)=t_{{\sf H}\bowtie {\sf H}}(g,b)=t_{\sf A}(b),$$
by (e1) of Definition \ref{mp}, $$ s_{\sf A}(\Psi_{1}(g,b))= s_{\sf A}(\varphi_{\sf A}(g,b))=t_{{\sf H}}(\phi_{\sf H}(g,b))=t_{{\sf H}}(\Psi_{2}(g,b))$$
holds, by (c1) and (d1) of Definition \ref{action}, we have the identities
$$t_{\sf A}(\Psi_{1}(g,b))=t_{\sf A}(\varphi_{\sf A}(g,b))=t_{\sf H}(g)=s_{\sf A}(a)=s_{\sf H}(\phi_{\sf H}(g,b))=s_{\sf H}(\Psi_{2}(g,b))$$
and, finally, 
$$s_{{\sf A}\bowtie {\sf H}}(a\bullet \Psi_{1}(g,b))=s_{{\sf A}}(\Psi_{1}(g,b))=t_{{\sf H}}(\Psi_{2}(g,b))=t_{{\sf A}\bowtie {\sf H}}(\Psi_{2}(g,b)\star h).$$

Also, the inverse map is well defined because, by (\ref{E-1}) and (\ref{E-2}), the pair $(a,h)$ belongs to the set $({\sf A}\bowtie {\sf H})_{1}$ iff $(\lambda_{\sf H}(h),\lambda_{\sf A}(a))\in  {\sf H}_1\;_{s_{\sf H}}\hspace{-0.15cm}\times_{t_{\sf A}} {\sf A}_1$. 

Let $(a,g)\in ({\sf A}\bowtie {\sf H})_{1}$. Then, in one hand 
\begin{itemize}
\item[ ]$\hspace{0.38cm} id_{{\sf A}\bowtie {\sf H}}(t_{{\sf A}\bowtie {\sf H}}(a,g))._{\Psi} (a,g) $
\item [ ]$=id_{{\sf A}\bowtie {\sf H}}(t_{{\sf A}}(a))._{\Psi} (a,g) $ {\scriptsize (by definition of $t_{{\sf A}\bowtie {\sf H}}$)}
\item [ ]$=(id_{{\sf A}}(t_{{\sf A}}(a)),id_{{\sf H}}(t_{{\sf A}}(a)))._{\Psi} (a,g) $ {\scriptsize (by definition of $id_{{\sf A}\bowtie {\sf H}}$)}
\item [ ]$=(id_{{\sf A}}(t_{{\sf A}}(a))\bullet \varphi_{\sf A}(id_{{\sf H}}(t_{{\sf A}}(a)),a),\phi_{\sf H}(id_{{\sf H}}(t_{{\sf A}}(a)),a)\star g ) $ {\scriptsize (by definition of $\Psi$)}
\item [ ]$=(\varphi_{\sf A}(id_{{\sf H}}(t_{{\sf A}}(a)),a),id_{\sf H}(s_{\sf A}(a))\star g ) $ {\scriptsize (by (c1) of Definition \ref{action},  (${\rm a2-1}$) of Definition \ref{quasigroupoid} and (\ref{P-2}))}
\item [ ]$=(a,id_{\sf H}(t_{\sf H}(g))\star g ) $ {\scriptsize (by (c3) of Definition \ref{action} and $(a,g)\in {\sf A}_1\;_{s_{\sf A}}\hspace{-0.15cm}\times_{t_{\sf H}} {\sf H}_1$)}
\item [ ]$=(a, g )$ {\scriptsize (by (a2-1) of Definition \ref{quasigroupoid})}
\end{itemize}
and, on the other hand, the proof for $(a,g)._{\Psi} id_{{\sf A}\bowtie {\sf H}}(s_{{\sf A}\bowtie {\sf H}}(a,g))= (a, g )$ follows by a similar calculus.
Thus,  (a2-1) of Definition \ref{quasigroupoid} holds for ${\sf A}\bowtie {\sf H}$. Moreover, 
$$s_{{\sf A}\bowtie {\sf H}} ((a,g)._{\Psi}(b,h))=s_{{\sf A}\bowtie {\sf H}} ((a\bullet \varphi_{\sf A}(g,b),\phi_{\sf H}(g,b)\star h)=s_{{\sf H}}(\phi_{\sf H}(g,b)\star h)=s_{{\sf H}}(h)= s_{{\sf A}\bowtie {\sf H}}(b,h)$$
and, similarly, $t_{{\sf A}\bowtie {\sf H}} ((a,g)._{\Psi}(b,h))=t_{{\sf A}\bowtie {\sf H}} ((a,g))$. Then, we have that  (${\rm a 2-2}$) of Definition \ref{quasigroupoid} also holds for ${\sf A}\bowtie {\sf H}$.

Let $ ((a,g),(b,h))\in ({\sf A}\bowtie {\sf H})_{1}\;_{s_{\sf {\sf A}\bowtie {\sf H}}}\hspace{-0.15cm}\times_{t_{{\sf A}\bowtie {\sf H}}} ({\sf A}\bowtie {\sf H})_{1}$. Then, by (d1) of Definition \ref{action}, 
\begin{align*}
	s_{{\sf A}\bowtie {\sf H}}(\lambda_{{\sf A}\bowtie {\sf H}}(a,g))& =s_{{\sf A}\bowtie {\sf H}}(\varphi_{\sf A}(\lambda_{\sf H}(g), \lambda_{\sf A}(a)), \phi_{\sf H}(\lambda_{\sf H}(g), \lambda_{\sf A}(a)))\\
	\;	& =s_{{\sf H}}(\phi_{\sf H}(\lambda_{\sf H}(g), \lambda_{\sf A}(a)))\\
	\;	& =(s_{{\sf A}}(\lambda_{\sf A}(a))\\
	\;	& = t_{{\sf A}}(a)\\
	\;	& = t_{{\sf A}}(a\bullet \varphi_{\sf A}(g,b))\\
	\;	& =t_{{\sf A}\bowtie {\sf H}}((a,g)._{\Psi}(b,h))
\end{align*}
hold and this implies that  $(\lambda_{{\sf A}\bowtie {\sf H}}(a,g),(a,g)._{\Psi}(b,h))\in ({\sf A}\bowtie {\sf H})_{1}\;_{s_{\sf {\sf A}\bowtie {\sf H}}}\hspace{-0.15cm}\times_{t_{{\sf A}\bowtie {\sf H}}} ({\sf A}\bowtie {\sf H})_{1}$. Similarly, we can prove that $((a,g)._{\Psi}(b,h), \lambda_{{\sf A}\bowtie {\sf H}}(b,h))\in ({\sf A}\bowtie {\sf H})_{1}\;_{s_{\sf {\sf A}\bowtie {\sf H}}}\hspace{-0.15cm}\times_{t_{{\sf A}\bowtie {\sf H}}} ({\sf A}\bowtie {\sf H})_{1}$. 
Finally, 
$$ \lambda_{{\sf A}\bowtie {\sf H}}(a,g)._{\Psi}((a,g)._{\Psi}(b,h))= $$
$$(\varphi_{\sf A}(\lambda_{\sf H}(g), \lambda_{\sf A}(a))\bullet  \varphi_{\sf A}(\phi_{\sf H}(\lambda_{\sf H}(g), \lambda_{\sf A}(a)), a\bullet \varphi_{\sf A}(g,b)), 
\phi_{\sf H}(\phi_{\sf H}(\lambda_{\sf H}(g), \lambda_{\sf A}(a)), a\bullet \varphi_{\sf A}(g,b))\star (\phi_{\sf H}(g,b)\star h ))$$ 
where 
\begin{itemize}
\item[ ]$\hspace{0.38cm} \varphi_{\sf A}(\lambda_{\sf H}(g), \lambda_{\sf A}(a))\bullet  \varphi_{\sf A}(\phi_{\sf H}(\lambda_{\sf H}(g), \lambda_{\sf A}(a)), a\bullet \varphi_{\sf A}(g,b)) $
\item [ ]$=\varphi_{\sf A}(\lambda_{\sf H}(g), \lambda_{\sf A}(a))\bullet  (\varphi_{\sf A}(\phi_{\sf H}(\lambda_{\sf H}(g), \lambda_{\sf A}(a)),a)\bullet \varphi_{\sf A}(\phi_{\sf H}(\phi_{\sf H}(\lambda_{\sf H}(g), \lambda_{\sf A}(a)),a), \varphi_{\sf A}(g,b))) $ {\scriptsize (by}
\item[ ]$\hspace{0.38cm}$  {\scriptsize (e2) of  Definition \ref{mp})}
\item [ ]$=\varphi_{\sf A}(\lambda_{\sf H}(g), \lambda_{\sf A}(a))\bullet  (\varphi_{\sf A}(\phi_{\sf H}(\lambda_{\sf H}(g), \lambda_{\sf A}(a)),\lambda_{\sf A}(\lambda_{\sf A}(a))) $ 
\item[ ]$\hspace{0.38cm}\bullet \varphi_{\sf A}(\phi_{\sf H}(\phi_{\sf H}(\lambda_{\sf H}(g), \lambda_{\sf A}(a)),a), \varphi_{\sf A}(g,b)))$  {\scriptsize (by (\ref{E-5}))}
\item [ ]$=\lambda_{\sf A}(\lambda_{\sf A}(\varphi_{\sf A}(\lambda_{\sf H}(g), \lambda_{\sf A}(a))))\bullet  (\lambda_{\sf A}(\varphi_{\sf A}(\lambda_{\sf H}(g), \lambda_{\sf A}(a))) \bullet \varphi_{\sf A}(\phi_{\sf H}(\phi_{\sf H}(\lambda_{\sf H}(g), \lambda_{\sf A}(a)),a), \varphi_{\sf A}(g,b)))$
\item[ ]$\hspace{0.38cm}$  {\scriptsize (by (\ref{E-5}) and  (\ref{P-3}))}
\item [ ]$= \varphi_{\sf A}(\phi_{\sf H}(\phi_{\sf H}(\lambda_{\sf H}(g), \lambda_{\sf A}(a)),a), \varphi_{\sf A}(g,b))$ {\scriptsize (by (${\rm a2-3}$) of Definition  \ref{quasigroupoid})}
\item [ ]$=\varphi_{\sf A}(\phi_{\sf H}(\lambda_{\sf H}(g), \lambda_{\sf A}(a)\bullet a), \varphi_{\sf A}(g,b)) $ {\scriptsize (by (d2) of Definition  \ref{action})}
\item [ ]$=\varphi_{\sf A}(\phi_{\sf H}(\lambda_{\sf H}(g), id_{\sf A}(s_{\sf A}(a))), \varphi_{\sf A}(g,b)) $ {\scriptsize (by (\ref{E-3}))}
\item [ ]$=\varphi_{\sf A}(\phi_{\sf H}(\lambda_{\sf H}(g), id_{\sf A}(t_{\sf H}(g))), \varphi_{\sf A}(g,b))  $ {\scriptsize (by $(a,g)\in {\sf A}_1\;_{s_{\sf A}}\hspace{-0.15cm}\times_{t_{\sf H}} {\sf H}_1$)}
\item [ ]$=\varphi_{\sf A}(\phi_{\sf H}(\lambda_{\sf H}(g), id_{\sf A}(s_{\sf H}(\lambda_{\sf H}(g)))), \varphi_{\sf A}(g,b))  $ {\scriptsize (by (\ref{E-1}))}
\item [ ]$=\varphi_{\sf A}(\lambda_{\sf H}(g), \varphi_{\sf A}(g,b))  $ {\scriptsize (by (d3) of Definition  \ref{action})}
\item [ ]$=\varphi_{\sf A}(\lambda_{\sf H}(g)\star g, b)  $ {\scriptsize (by (c2) of Definition  \ref{action})}
\item [ ]$=\varphi_{\sf A}(id_{\sf H}(s_{\sf H}(g)), b) $ {\scriptsize (by (\ref{E-1}))}
\item [ ]$= \varphi_{\sf A}(id_{\sf H}(t_{\sf A}(b)), b)$ {\scriptsize (by $s_{\sf H}(g) =s_{{\sf A}\bowtie {\sf H}}(a,g)=t_{{\sf A}\bowtie {\sf H}}(b,h)=t_{\sf A}(b)$)}
\item [ ]$=b $ {\scriptsize (by (c3) of Definition  \ref{action})}
\end{itemize}
and 
\begin{itemize}
\item[ ]$\hspace{0.38cm} \phi_{\sf H}(\phi_{\sf H}(\lambda_{\sf H}(g), \lambda_{\sf A}(a)), a\bullet \varphi_{\sf A}(g,b))\star (\phi_{\sf H}(g,b)\star h ) $
\item [ ]$=\phi_{\sf H}(\phi_{\sf H}(\phi_{\sf H}(\lambda_{\sf H}(g), \lambda_{\sf A}(a)), a),  \varphi_{\sf A}(g,b))\star (\phi_{\sf H}(g,b)\star h )  $ {\scriptsize (by (d2) of Definition  \ref{action})}
\item [ ]$=\phi_{\sf H}(\phi_{\sf H}(\lambda_{\sf H}(g), \lambda_{\sf A}(a)\bullet a),  \varphi_{\sf A}(g,b))\star (\phi_{\sf H}(g,b)\star h )  $ {\scriptsize (by (d2) of Definition  \ref{action})}
\item [ ]$=\phi_{\sf H}(\phi_{\sf H}(\lambda_{\sf H}(g), id_{\sf A}(s_{\sf A}(a))),  \varphi_{\sf A}(g,b))\star (\phi_{\sf H}(g,b)\star h ) $ {\scriptsize (by (\ref{E-3}))}
\item [ ]$= \phi_{\sf H}(\phi_{\sf H}(\lambda_{\sf H}(g), id_{\sf A}(t_{\sf H}(g))),  \varphi_{\sf A}(g,b))\star (\phi_{\sf H}(g,b)\star h )$ {\scriptsize (by $(a,g)\in {\sf A}_1\;_{s_{\sf A}}\hspace{-0.15cm}\times_{t_{\sf H}} {\sf H}_1$)}
\item [ ]$=\phi_{\sf H}(\phi_{\sf H}(\lambda_{\sf H}(g), id_{\sf A}(s_{\sf H}(\lambda_{\sf H}(g)))),  \varphi_{\sf A}(g,b))\star (\phi_{\sf H}(g,b)\star h ) $ {\scriptsize (by (\ref{E-1}))}
\item [ ]$=\phi_{\sf H}(\lambda_{\sf H}(g),  \varphi_{\sf A}(g,b))\star (\phi_{\sf H}(g,b)\star h ) $ {\scriptsize (by (d3) of Definition  \ref{action})}
\item [ ]$=(\phi_{\sf H}(\lambda_{\sf H}(g)\star g, b)\star \lambda_{\sf H}(\phi_{\sf H}(g,b)))\star (\phi_{\sf H}(g,b)\star h )  $ {\scriptsize (by (e3) of Definition  \ref{mp} and (${\rm a2-3}$) of }
\item[ ]$\hspace{0.38cm}$ {\scriptsize Definition \ref{quasigroupoid})}
\item [ ]$=(\phi_{\sf H}(id_{\sf H}(s_{\sf H}(g)), b)\star \lambda_{\sf H}(\phi_{\sf H}(g,b)))\star (\phi_{\sf H}(g,b)\star h ) $ {\scriptsize (by (\ref{E-3}))}
\item [ ]$=(\phi_{\sf H}(id_{\sf H}(t_{\sf A}(b)),b)\star \lambda_{\sf H}(\phi_{\sf H}(g,b)))\star (\phi_{\sf H}(g,b)\star h ) $ {\scriptsize (by $s_{\sf H}(g) =s_{{\sf A}\bowtie {\sf H}}(a,g)=t_{{\sf A}\bowtie {\sf H}}(b,h)=t_{\sf A}(b)$)}
\item [ ]$=(id_{\sf H}(s_{\sf A}(b))\star \lambda_{\sf H}(\phi_{\sf H}(g,b)))\star (\phi_{\sf H}(g,b)\star h ) $ {\scriptsize (by (\ref{P-2}))}
\item [ ]$= \lambda_{\sf H}(\phi_{\sf H}(g,b))\star (\phi_{\sf H}(g,b)\star h )  $ {\scriptsize (by (d1) of Definition  \ref{action} and (${\rm a2-1}$) of Definition \ref{quasigroupoid})}
\item [ ]$= h   $ {\scriptsize (by (${\rm a2-3}$) of Definition \ref{quasigroupoid})}.
\end{itemize}

Thus,  $\lambda_{{\sf A}\bowtie {\sf H}}(a,g)._{\Psi}((a,g)._{\Psi}(b,h))=(b,h)$ and, by a similar proof, 
we obtain that $$((a,g)._{\Psi}(b,h))._{\Psi}\lambda_{{\sf A}\bowtie {\sf H}}(b,h)=(a,g).$$  Therefore, (a2-3) of Definition \ref{quasigroupoid}  holds for ${\sf A}\bowtie {\sf H}$ and we can assure that ${\sf A}\bowtie {\sf H}$  is a quasigroupoid.
\end{proof}

\begin{definition}
\label{dbcp}
{\rm Let  $({\sf A}, {\sf H})$ be a matched pair of quasigroupoids. The quasigroupoid ${\sf A}\bowtie {\sf H}$ will be called the double crossed product or the diagonal quasigroupoid of ${\sf A}$ and ${\sf H}$.
}
\end{definition}

In the following lines we will give two simple examples of matched pair of quasigroupoids.

\begin{example}
\label{ex-mp}
{\rm  Let ${\sf A}$ be a quasigroupoid and denote by $X$ the base set of ${\sf A}$.  Then, $({\sf A}, {\sf X}^d)$ is an example of matched pair of quasigroupoids  where $\varphi_{\sf A}(x,a)=a$ and 
$\phi_{{\sf X}^d}(x,a)=s_{\sf A}(a)$. Then, the diagonal quasigroupoid ${\sf A}\bowtie {\sf X}^d$, is defined by $$({\sf A}\bowtie {\sf X}^d)_{0}={\sf A}_{0}, \;\;\;({\sf A}\bowtie {\sf X}^d)_{1}= {\sf A}_1\;_{s_{\sf A}}\hspace{-0.15cm}\times_{t_{{\sf X}^d}} {\sf X}^d_1$$
$$s_{{\sf A}\bowtie {{\sf X}^d}}(a,x)=x, \;\;\;t_{{\sf A}\bowtie {{\sf X}^d}}(a,x)=t_{\sf A}(a), \;\;\;id_{{\sf A}\bowtie {\sf X}^d}(x)=
(id_{\sf A}(x), x),$$
$$(a,x)._{\Psi}(b,y)=(a\bullet b,y )$$
and $\lambda_{{\sf A}\bowtie {\sf X}^d}(a,x)=(\lambda_{\sf A}(a), t_{\sf A}(a)).$
}
\end{example}

\begin{example}
\label{ex-mp1}
{\rm
Let $A$ be a quasigroup with product $\cdot$ and let $X$ be a set. Assume that there exists an action of $A$ over $X$ denoted by $\psi_{X}:A\times X\rightarrow X$. Let   ${\sf B}=({\sf B}_0,{\sf B}_1)$ be the quasigroupoid associated to the action $\psi_{X}$  (see Example \ref{exquasi}). Let  ${\sf X}^{d}$ be the discrete groupoid. Then, $({\sf X}^{d}, {\sf B})$ is a matched pair of quasigroupoids where $\varphi_{{\sf X}^{d}}((a,x), x)=\psi_{X}(a,x)$ and $\phi_{{\sf B}}((a,x), x)=(a,x)$. In this case the diagonal quasigroupoid ${\sf X}^{d}\bowtie {\sf B}$, is defined by $$({\sf X}^{d}\bowtie {\sf B})_{0}=X, \;\;\; ({\sf X}^{d}\bowtie {\sf B})_{1}= {\sf X}^d_1\;_{s_{{\sf X}^d}}\hspace{-0.15cm}\times_{t_{{\sf B}}} {\sf B}_1,$$ 
$$s_{{\sf X}^{d}\bowtie {\sf B}}(y,(a,x))=x, \;\;\;t_{{\sf X}^{d}\bowtie {\sf B}}(y,(a,x))=y, \;\;\;id_{{\sf X}^{d}\bowtie {\sf B}}(x)=
(x, (e_{A},x)),$$ $$(y,(a,x))._{\Psi}(x,(b,t))=(y, (a\cdot b, t))$$
and $$\lambda_{{\sf X}^{d}\bowtie {\sf B}}(y,(a,x))=(\psi_{X}(a^{-1},y),(a^{-1},y)).$$

}
\end{example}

\begin{theorem}
\label{sub}
Let  $({\sf A}, {\sf H})$ be a matched pair of quasigroupoids. Then, ${\sf A}$ and ${\sf H}$ are subquasigroupoids of ${\sf A}\bowtie {\sf H}$.
\end{theorem}

\begin{proof} Define  $i^{\sf A}: {\sf A}\rightarrow {\sf A}\bowtie {\sf H}$ as $i^{\sf A}=(i^{\sf A}_{0}:{\sf A}_{0}\rightarrow ({\sf A}\bowtie {\sf H})_{0}, i^{\sf A}_{1}:{\sf A}_{1}\rightarrow ({\sf A}\bowtie {\sf H})_{1})$ where 
$$i^{\sf A}_{0}=id_{{\sf A}_{0}}, \;\;\; i^{\sf A}_{1}(a)=(a, id_{\sf H}(s_{\sf A}(a))).$$

Then, 
$$i^{\sf A}_{0}(s_{\sf A}(a))=s_{\sf A}(a)=s_{{\sf A}\bowtie {\sf H}}(a, id_{\sf H}(s_{\sf A}(a)))=s_{{\sf A}\bowtie {\sf H}}( i^{\sf A}_{1}(a))$$
and 
$$i^{\sf A}_{0}(t_{\sf A}(a))=t_{\sf A}(a)=t_{{\sf A}\bowtie {\sf H}}(a, id_{\sf H}(t_{\sf A}(a)))=t_{{\sf A}\bowtie {\sf H}}( i^{\sf A}_{1}(a)).$$

Also, for all $x\in {\sf A_{0}}$, 
$$i^{\sf A}_{1}(id_{\sf A}(x))=(id_{\sf A}(x), id_{\sf H}(x))=id_{{\sf A}\bowtie {\sf H}}(x)=id_{{\sf A}\bowtie {\sf H}}(i^{\sf A}_{0}(x))$$
holds. 

Therefore, (b1), (b2) and (b3) of  Definition \ref{mor-loop} hold. Finally, if $(a,b)\in {\sf A}_1\;_{s_{\sf A}}\hspace{-0.15cm}\times_{t_{\sf A}} {\sf A}_1$ we have that $s_{{\sf A}\bowtie {\sf H}}( i^{\sf A}_{1}(a))=s_{\sf A}(a)=t_{\sf A}(b)=t_{{\sf A}\bowtie {\sf H}}( i^{\sf A}_{1}(b))$ and this implies that there exists  the product $i^{\sf A}_{1}(a)._{\Psi}i^{\sf A}_{1}(b)$.  Then, in one hand, 
$$i^{\sf A}_{1}(a\bullet b)=(a\bullet b, id_{\sf H}(s_{\sf A}(a\bullet b)) )=(a\bullet b, id_{\sf H}(s_{\sf A}( b)))$$
and, on the other hand, 
\begin{itemize}
\item[ ]$\hspace{0.38cm} i^{\sf A}_{1}(a)._{\Psi}i^{\sf A}_{1}(b)$
\item [ ]$= (a, id_{\sf H}(s_{\sf A}(a)))._{\Psi}(b, id_{\sf H}(s_{\sf A}(b)))   $ {\scriptsize (by Definition of $i^{\sf A}_{1}$)}.
\item [ ]$= (a\bullet \varphi_{\sf A}(id_{\sf H}(s_{\sf A}(a)), b),  \phi_{\sf H}(id_{\sf H}(s_{\sf A}(a)), b)\star id_{\sf H}(s_{\sf A}(b)))  $ {\scriptsize (by Definition of $._{\Psi}$)}
\item [ ]$= (a\bullet \varphi_{\sf A}(id_{\sf H}(t_{\sf A}(b)), b),  \phi_{\sf H}(id_{\sf H}(t_{\sf A}(b)), b)\star id_{\sf H}(s_{\sf A}(b)))  $ {\scriptsize (by $(a,b)\in {\sf A}_1\;_{s_{\sf A}}\hspace{-0.15cm}\times_{t_{\sf A}} {\sf A}_1$)}.
\item [ ]$= (a\bullet b,  id_{\sf H}(s_{\sf A}(b))) $ {\scriptsize (by (c3)  of Definition \ref{action} and (\ref{P-2}))}.
\end{itemize}

As a consequence,  (b4) of  Definition \ref{mor-loop} holds and ${\sf A}$ is a subquasigroupoid of ${\sf A}\bowtie {\sf H}$. Similarly, if we define  $i^{\sf H}: {\sf H}\rightarrow {\sf A}\bowtie {\sf H}$ as $i^{\sf H}=(i^{\sf H}_{0}:{\sf A}_{0}\rightarrow ({\sf A}\bowtie {\sf H})_{0}, i^{\sf H}_{1}:{\sf H}_{1}\rightarrow ({\sf A}\bowtie {\sf H})_{1})$ where 
$$i^{\sf H}_{0}=id_{{\sf A}_{0}}, \;\;\; i^{\sf H}_{1}(g)=( id_{\sf A}(t_{\sf H}(g)), g),$$
it is possible to prove that ${\sf H}$ is a subquasigroupoid of ${\sf A}\bowtie {\sf H}$. 
\end{proof}

\begin{lemma}
\label{lem1}
Let  $({\sf A}, {\sf H})$ be a matched pair of quasigroupoids. Let $i^{\sf A}$, $i^{\sf H}$ be the monomorphisms defined in the previous theorem. 
\begin{itemize}
\item[(i)] If $(g,a)\in {\sf H}_1\;_{s_{\sf H}}\hspace{-0.15cm}\times_{t_{\sf A}} {\sf A}_1$ and $(a,b)\in {\sf A}_1\;_{s_{\sf A}}\hspace{-0.15cm}\times_{t_{\sf A}} {\sf A}_1$, the following equality holds:
\begin{equation}
\label{HAA} i^{\sf H}_1(g)._{\Psi}(i^{\sf A}_1(a)._{\Psi}i^{\sf A}_1(b))=(i^{\sf H}_1(g)._{\Psi}i^{\sf A}_1(a))._{\Psi}i^{\sf A}_1(b).
\end{equation}
\item[(ii)] If $(g,h)\in {\sf H}_1\;_{s_{\sf H}}\hspace{-0.15cm}\times_{t_{\sf H}} {\sf H}_1$ and $(h,a)\in {\sf H}_1\;_{s_{\sf H}}\hspace{-0.15cm}\times_{t_{\sf A}} {\sf A}_1$, the following equality holds:
\begin{equation}
\label{HHA} i^{\sf H}_1(g)._{\Psi}(i^{\sf H}_1(h)._{\Psi}i^{\sf A}_1(a))=(i^{\sf H}_1(g)._{\Psi}i^{\sf H}_1(h))._{\Psi}i^{\sf A}_1(a).
\end{equation}
\item[(iii)] If $(h,a)\in {\sf H}_1\;_{s_{\sf H}}\hspace{-0.15cm}\times_{t_{\sf A}} {\sf A}_1$ and $(a,f)\in {\sf A}_1\;_{s_{\sf A}}\hspace{-0.15cm}\times_{t_{\sf H}} {\sf H}_1$, the following equality holds:
\begin{equation}
	\label{HAH} i^{\sf H}_1(h)._{\Psi}(i^{\sf A}_1(a)._{\Psi}i^{\sf H}_1(f))=(i^{\sf H}_1(h)._{\Psi}i^{\sf A}_1(a))._{\Psi}i^{\sf H}_1(f).
\end{equation}
\item[(iv)] If $(c,h)\in {\sf A}_1\;_{s_{\sf A}}\hspace{-0.15cm}\times_{t_{\sf H}} {\sf A}_1$ and $(h,a)\in {\sf H}_1\;_{s_{\sf H}}\hspace{-0.15cm}\times_{t_{\sf A}} {\sf A}_1$, the following equality holds:
\begin{equation}
\label{AHA} i^{\sf A}_1(c)._{\Psi}(i^{\sf H}_1(h)._{\Psi}i^{\sf A}_1(a))=(i^{\sf A}_1(c)._{\Psi}i^{\sf H}_1(h))._{\Psi}i^{\sf A}_1(a).
\end{equation}
\item[(v)] If $(a,b)\in {\sf A}_1\;_{s_{\sf A}}\hspace{-0.15cm}\times_{t_{\sf A}} {\sf A}_1$ and $(b,g)\in {\sf A}_1\;_{s_{\sf A}}\hspace{-0.15cm}\times_{t_{\sf H}} {\sf H}_1$, the following equality holds:
\begin{equation}
\label{AAH} i^{\sf A}_1(a)._{\Psi}(i^{\sf A}_1(b)._{\Psi}i^{\sf H}_1(g))=(i^{\sf A}_1(a)._{\Psi}i^{\sf A}_1(b))._{\Psi}i^{\sf H}_1(g).
\end{equation}
\item[(vi)] If $(a,g)\in {\sf A}_1\;_{s_{\sf A}}\hspace{-0.15cm}\times_{t_{\sf H}} {\sf H}_1$ and $(g,h)\in {\sf H}_1\;_{s_{\sf H}}\hspace{-0.15cm}\times_{t_{\sf H}} {\sf H}_1$, the following equality holds:
\begin{equation}
\label{AHH} i^{\sf A}_1(a)._{\Psi}(i^{\sf H}_1(g)._{\Psi}i^{\sf H}_1(h))=(i^{\sf A}_1(a)._{\Psi}i^{\sf H}_1(h))._{\Psi}i^{\sf H}_1(g).
\end{equation}

\end{itemize}
\end{lemma}

\begin{proof} We will prove (i). The proofs for (ii)-(vi)  are similar and we left the details to the reader. 
	
Let $(g,a)\in {\sf H}_1\;_{s_{\sf H}}\hspace{-0.15cm}\times_{t_{\sf A}} {\sf A}_1$ and $(a,b)\in {\sf A}_1\;_{s_{\sf A}}\hspace{-0.15cm}\times_{t_{\sf A}} {\sf A}_1$. Then, in one hand
\begin{itemize}
\item[ ]$\hspace{0.38cm} i^{\sf H}_1(g)._{\Psi}(i^{\sf A}_1(a)._{\Psi}i^{\sf A}_1(b))$
\item [ ]$=(id_{\sf A}(t_{\sf H}(g)), g)._{\Psi}((a\bullet \varphi_{\sf A}(id_{\sf H}(s_{\sf A}(a)), b), \phi_{\sf H}(id_{\sf H}(s_{\sf A}(a)),b)\star id_{\sf H}(s_{\sf A}(b)))) $ {\scriptsize (by definitions of $i^{\sf A}$,}
\item[ ]$\hspace{0.38cm}$ {\scriptsize  $i^{\sf H}$ and $._{\Psi}$)}
\item [ ]$=(id_{\sf A}(t_{\sf H}(g)), g)._{\Psi} (a\bullet b,id_{\sf H}(s_{\sf A}(b))) $ {\scriptsize (by $s_{\sf A}(a)=t_{\sf A}(b)$, (c3) of definition \ref{action} and (\ref{P-2}))}
\item [ ]$=(id_{\sf A}(t_{\sf H}(g))\bullet \varphi_{\sf A}(g, a\bullet b),\phi_{\sf H}(g, a\bullet b)\star id_{\sf H}(s_{\sf A}(b))) $ {\scriptsize (by definition of $._{\Psi}$)}
\item [ ]$=(\varphi_{\sf A}(g, a\bullet b),\phi_{\sf H}(g, a\bullet b)) $ {\scriptsize (by (c1) and (d1) of Definition \ref{action} and  (${\rm a2-2}$) and (${\rm a2-1}$) of Definition \ref{quasigroupoid})}
\end{itemize}	
and on the other hand, 
\begin{itemize}
\item[ ]$\hspace{0.38cm}(i^{\sf H}_1(g)._{\Psi}i^{\sf A}_1(a))._{\Psi}i^{\sf A}_1(b)$
\item [ ]$=(id_{\sf A}(t_{\sf H}(g))\bullet \varphi_{\sf A}(g,a),\phi_{\sf H}(g,a)\star id_{\sf H}(s_{\sf A}(a)))._{\Psi}(b,id_{\sf H}(s_{\sf A}(b))) $ {\scriptsize (by definitions of $i^{\sf A}$, $i^{\sf H}$ and $._{\Psi}$)}
\item [ ]$=(\varphi_{\sf A}(g,a),\phi_{\sf H}(g,a))._{\Psi}(b,id_{\sf H}(s_{\sf A}(b))) $ {\scriptsize (by (c1) and  (d1) of Definition \ref{action} and (${\rm a2-1}$) of Definition \ref{quasigroupoid})}
\item [ ]$=(\varphi_{\sf A}(g,a)\bullet \varphi_{\sf A}(\phi_{\sf H}(g,a),b), \phi_{\sf H}(\phi_{\sf H}(g,a),b)\star id_{\sf H}(s_{\sf A}(b))) $ {\scriptsize (by definition of $._{\Psi}$)}
\item [ ]$=(\varphi_{\sf A}(g, a\bullet b),\phi_{\sf H}(g, a\bullet b)) $ {\scriptsize (by (e2) of Definition \ref{mp}, (d1) and (d2) of Definition \ref{action} and   (a2-1) of}
\item[ ]$\hspace{0.38cm}$ {\scriptsize  Definition \ref{quasigroupoid})}
\end{itemize}

Therefore (\ref{HAA}) holds. 
\end{proof}

\begin{lemma}
\label{lem2}
Let  $({\sf A}, {\sf H})$ be a matched pair of quasigroupoids. Let $i^{\sf A}$, $i^{\sf H}$ be the monomorphisms defined in Theorem \ref{sub}.The map $$\theta_{{\sf A}\bowtie {\sf H}}:{\sf A}_1\;_{s_{\sf A}}\hspace{-0.15cm}\times_{t_{\sf H}} {\sf H}_1\rightarrow {\sf A}_1\;_{s_{\sf A}}\hspace{-0.15cm}\times_{t_{\sf H}} {\sf H}_1$$ defined by 
$$\theta_{{\sf A}\bowtie {\sf H}}=._{\Psi}\circ (i^{\sf A}_{1}\times i^{\sf H}_{1})$$ is the identity.
\end{lemma}

\begin{proof}
First note that, if $(a,g)\in{\sf A}_1\;_{s_{\sf A}}\hspace{-0.15cm}\times_{t_{\sf H}} {\sf H}_1$, we have that 
$$s_{{\sf A}\bowtie {\sf H}} (i^{\sf A}_{1}(a))=s_{{\sf A}\bowtie {\sf H}} (a, id_{\sf H}(s_{\sf A}(a)))=s_{\sf A}(a)=t_{\sf H}(g)=t_{{\sf A}\bowtie {\sf H}} (id_{\sf A}(t_{\sf H}(g)), g)=t_{{\sf A}\bowtie {\sf H}} (i^{\sf H}_{1}(g)).$$

Then, there exists $\theta_{{\sf A}\bowtie {\sf H}}(a,g)=i^{\sf A}_{1}(a)._{\Psi}i^{\sf H}_{1}(g)$ and, by (c3) and (d3) of Definition \ref{action} and (a2-1) of Definition \ref{quasigroupoid}, we have that 
\begin{align*}
\theta_{{\sf A}\bowtie {\sf H}}(a,g)& =(a, id_{\sf H}(s_{\sf A}(a)))._{\Psi} ( id_{\sf A}(t_{\sf H}(g)), g)\\
\;	& =(a\bullet \varphi_{\sf A}( id_{\sf H}(s_{\sf A}(a)),  id_{\sf A}(t_{\sf H}(g))), \phi_{\sf H}( id_{\sf H}(s_{\sf A}(a)),  id_{\sf A}(t_{\sf H}(g)))\star g)\\
\;	& =(a\bullet \varphi_{\sf A}( id_{\sf H}(s_{\sf A}(a)),  id_{\sf A}(s_{\sf A}(a))), \phi_{\sf H}( id_{\sf H}(t_{\sf H}(g)),  id_{\sf A}(t_{\sf H}(g)))\star g)\\
\;	& = (a,g).
\end{align*}
\end{proof}

\begin{definition}
\label{exact-f}
{\rm Let {\sf A}, {\sf H} and {\sf B} be quasigroupoids with the same base. Assume that {\sf A} and  {\sf H} are subquasigroupoids of  {\sf B} with associated monomorphisms $i^{\sf A}: {\sf A}\rightarrow {\sf B}$ and $i^{\sf H}: {\sf H}\rightarrow {\sf B}$  satisfying that $i^{\sf A}_{0}=i^{\sf H}_{0}=id_{{\sf A}_{0}}$. Let $\diamond$ be the product defined in {\sf B}. Then we will say that $[{\sf A}, {\sf H}]$ is an exact factorization of {\sf B} if :
\begin{itemize}
\item[(i)]  For all $(g,a)\in {\sf H}_1\;_{s_{\sf H}}\hspace{-0.15cm}\times_{t_{\sf A}} {\sf A}_1$ and $(a,b)\in {\sf A}_1\;_{s_{\sf A}}\hspace{-0.15cm}\times_{t_{\sf A}} {\sf A}_1$, the equality (\ref{HAA}) holds for $._{\Psi}=\diamond$.
\item[(ii)]  For all $(g,h)\in {\sf H}_1\;_{s_{\sf H}}\hspace{-0.15cm}\times_{t_{\sf H}} {\sf H}_1$ and $(h,a)\in {\sf H}_1\;_{s_{\sf H}}\hspace{-0.15cm}\times_{t_{\sf A}} {\sf A}_1$, the equality (\ref{HHA}) holds for $._{\Psi}=\diamond$.
\item[(iii)] For all $(h,a)\in {\sf H}_1\;_{s_{\sf H}}\hspace{-0.15cm}\times_{t_{\sf A}} {\sf A}_1$ and $(a,f)\in {\sf A}_1\;_{s_{\sf A}}\hspace{-0.15cm}\times_{t_{\sf H}} {\sf H}_1$, the equality (\ref{HAH}) holds for $._{\Psi}=\diamond$.
\item[(iv)] For all $(c,h)\in {\sf A}_1\;_{s_{\sf A}}\hspace{-0.15cm}\times_{t_{\sf H}} {\sf A}_1$ and $(h,a)\in {\sf H}_1\;_{s_{\sf H}}\hspace{-0.15cm}\times_{t_{\sf A}} {\sf A}_1$, the equality (\ref{AHA}) holds for $._{\Psi}=\diamond$.
\item[(v)] For all  $(a,b)\in {\sf A}_1\;_{s_{\sf A}}\hspace{-0.15cm}\times_{t_{\sf A}} {\sf A}_1$ and $(b,g)\in {\sf A}_1\;_{s_{\sf A}}\hspace{-0.15cm}\times_{t_{\sf H}} {\sf H}_1$, the equality (\ref{AAH}) holds for $._{\Psi}=\diamond$.
\item[(vi)] For all $(a,g)\in {\sf A}_1\;_{s_{\sf A}}\hspace{-0.15cm}\times_{t_{\sf H}} {\sf H}_1$ and $(g,h)\in {\sf H}_1\;_{s_{\sf H}}\hspace{-0.15cm}\times_{t_{\sf H}} {\sf H}_1$, the equality (\ref{AHH}) holds for $._{\Psi}=\diamond$.
\item[(vii)] The map $\theta_{\sf B}=\diamond\circ (i^{\sf A}_1\times i^{\sf H}_1): {\sf A}_1\;_{s_{\sf A}}\hspace{-0.15cm}\times_{t_{\sf H}} {\sf H}_1\rightarrow {\sf B}_{1}$ is a bijection.
\end{itemize}
}
\end{definition}

\begin{example}
Let  $({\sf A}, {\sf H})$ be a matched pair of quasigroupoids. By Theorem \ref{sub} and Lemmas \ref{lem1} and \ref{lem2}, we have that $[{\sf A}, {\sf H}]$ is an exact factorization of ${\sf A}\bowtie {\sf H}$. Then any matched pair of quasigroupoids give rise to an exact factorization. 

In the following theorem we will prove that any exact factorization $[{\sf A}, {\sf H}]$ of a quasigropupoid  {\sf B} induces an example of matched pair of quasigroupoids.
\end{example}

\begin{theorem}
\label{main} Let  $[{\sf A}, {\sf H}]$ be an exact factorization of a quasigropupoid  {\sf B}. Then, there exists  a matched pair of quasigroupoids   $({\sf A}, {\sf H})$ and an isomorphism of quasigroupoids  between ${\sf A}\bowtie {\sf H}$ and ${\sf B}$. 
\end{theorem}	

\begin{proof} Let $(h,a)\in {\sf H}_1\;_{s_{\sf H}}\hspace{-0.15cm}\times_{t_{\sf A}} {\sf A}_1$. Then, $(i^{\sf H}_1(h),i^{\sf A}_1(a))\in {\sf B}_1\;_{s_{\sf B}}\hspace{-0.15cm}\times_{t_{\sf B}} {\sf B}_1$, because, by  (b1) and (b2) of Definition \ref{mor-loop}, we have that following:
$$s_{\sf B}(i^{\sf H}_{1}(h))=s_{\sf H}(h)=t_{\sf A}(a)=t_{\sf B}(i^{\sf A}_{1}(a)).$$

Therefore, thanks to the existence of the isomorphism $\theta_{B}$, we can guarantee the existence of a unique pair $(a^h, h^a)\in  {\sf A}_1\;_{s_{\sf A}}\hspace{-0.15cm}\times_{t_{\sf H}} {\sf H}_1$ such that 
\begin{equation}
\label{isopro}
i^{\sf H}_{1}(h)\diamond i^{\sf A}_{1}(a)=i^{\sf A}_{1}(a^h)\diamond i^{\sf H}_{1}(h^a).
\end{equation}

This last equality 	allows to define the  maps  
$$\varphi_{\sf A} :{\sf H}_1\;_{s_{\sf H}}\hspace{-0.15cm}\times_{t_{\sf A}} {\sf A}_1\rightarrow {\sf A}_1,\;\;\;\; \varphi_{\sf A}(h,a)= a^h,$$ 
and 
$$\phi_{\sf H} :{\sf H}_1\;_{s_{\sf H}}\hspace{-0.15cm}\times_{t_{\sf A}} {\sf A}_1\rightarrow {\sf A}_1,\;\;\;\; \phi_{\sf H}(h,a)= h^a.$$ 

The first map is a left action of $H$ on $A$. Indeed: First note that, by (b2) of Definition \ref{mor-loop} and (\ref{isopro}), the following identities 
$$t_{\sf A}(\varphi_{\sf A}(h,a))=t_{\sf A}(a^h)=t_{\sf B}(i^{\sf A}_{1}(a^h))=t_{\sf B}(i^{\sf H}_{1}(h))=t_{\sf H}(h)$$
hold and, as a consequence, we obtain (c1) of Definition \ref{action}. The proof of  (d1) of Definition \ref{action} follows a similar pattern and we left the details to the reader. Also, by   (\ref{isopro}), (\ref{HHA}),   (\ref{HAH}),  (\ref{AHH}) and (b4) of Definition \ref{mor-loop}, we have the identities:

\begin{align*}
	i^{\sf A}_{1}(a^{g\star h})\diamond i^{\sf H}_{1}((g\star h)^a)   & = 	i^{\sf H}_{1}(g\star h)\diamond i^{\sf A}_{1}(a)   \\
	\;	& = (i^{\sf H}_{1}(g)\diamond i^{\sf H}_{1}(h))\diamond i^{\sf A}_{1}(a)  \\
	\;	& =  i^{\sf H}_{1}(g)\diamond ( i^{\sf H}_{1}(h)\diamond i^{\sf A}_{1}(a) )\\
	\;	& = i^{\sf H}_{1}(g)\diamond (  i^{\sf A}_{1}(a^h)\diamond i^{\sf H}_{1}(h^a)) \\
	\;	& = ( i^{\sf H}_{1}(g)\diamond  i^{\sf A}_{1}(a^h))\diamond i^{\sf H}_{1}(h^a) \\
	\;	& = ( i^{\sf A}_{1}((a^h)^{g})\diamond  i^{\sf H}_{1}(g^{a^h}))\diamond i^{\sf H}_{1}(h^a) \\
	\;	& =  i^{\sf A}_{1}((a^h)^{g})\diamond  (i^{\sf H}_{1}(g^{a^h})\diamond i^{\sf H}_{1}(h^a) )\\
	\;	& =  i^{\sf A}_{1}((a^h)^{g})\diamond  (i^{\sf H}_{1}(g^{a^h}\star h^a)).\\
	\end{align*}

As a consequence, 
$$\varphi_{\sf A}(g\star h, a)=a^{g\star h}=(a^h)^{g}=\varphi_{\sf A}(g, \varphi_{\sf A}(h,a)), $$
i.e., (c2) of Definition \ref{action} holds and 
$$\phi_{\sf H}(g\star h, a)=(g\star h)^a=g^{a^h}\star h^a=\phi_{\sf H}(g, \varphi_{\sf A}(h,a))\star \phi_{\sf H}(h,a),$$
i.e. (e3) of Definition \ref{mp} holds.  Similarly, we can prove that (d2) of Definition \ref{action} and (e2) of Definition \ref{mp} hold. 

On the other hand,  $\varphi_{\sf A}(id_{\sf H}(t_{\sf A}(a)), a)=a$  because 
$$i^{\sf H}_{1}(id_{\sf H}(t_{\sf A}(a)))\diamond i^{\sf A}_{1}(a)= i^{\sf A}_{1}(a)=i^{\sf A}_{1}(a)\diamond i^{\sf H}_{1}(id_{\sf H}(s_{\sf A}(a))).$$
and, similarly, we can prove that   $\phi_{\sf H}(h, id_{\sf A}(s_{\sf H}(h)))=h$. Also,  if $(h,a)\in {\sf H}_1\;_{s_{\sf H}}\hspace{-0.15cm}\times_{t_{\sf A}} {\sf A}_1$, 
$$s_{\sf A}( \varphi_{\sf A}(h,a))=s_{\sf A}( a^h)=s_{\sf B}( i^{\sf A}_{1}(a^h))=t_{\sf B}( i^{\sf H}_{1}(h^a))
=t_{\sf H}( h^a)=t_{\sf H}( \phi_{\sf A}(h,a)).$$

Therefore, with the previous actions,  $({\sf A}, {\sf H})$  is a matched pair of quasigroupoids. Finally, using the definition of exact factorization, it is easy to prove that  the quasigroupoid isomorphism is defined by 
$$\Gamma=(\Gamma_{0}=id_{{\sf A}_{0}}, \Gamma_{1}=\theta_{\sf B}).$$
\end{proof}

Theorems \ref{match-1} and \ref{main}  extend to the non-associative context similar results for groupoids. In the associative context the details of the constructions can be found in \cite[Theorems 2.10 and 2.15]{Mac}. As was pointed by K. Mackenzie these two results are a  generalization of \cite[3.8]{LW}.

\begin{definition}
{\rm Let $({\sf A}, {\sf H})$ and  $({\sf A}^{\prime}, {\sf H}^{\prime})$ be a pair of matched pair of quasigroupoids. A morphism $({\sf A}, {\sf H})\rightarrow ({\sf A}^{\prime}, {\sf H}^{\prime})$  of matched pairs of quasigroupoids is a pair $(\Gamma, \Omega)$ of morphisms of quasigroupoids $\Gamma: {\sf A}\rightarrow {\sf A}^{\prime}$ and $\Omega: {\sf H}\rightarrow {\sf H}^{\prime}$ such that 
\begin{align*}
\Gamma_{1}\circ \varphi_{\sf A}	  & =   \varphi_{\sf A^{\prime}}\circ (\Omega_{1}\times \Gamma_{1}),\\
\Omega_{1}\circ \phi_{\sf H}	  & =   \phi_{\sf H^{\prime}}\circ (\Omega_{1}\times \Gamma_{1}).
\end{align*}

If  $({\sf A}^{\prime\prime }, {\sf H}^{\prime\prime})$ is another matched pair of quasigroupoids and $(\Gamma^{\prime}, \Omega^{\prime}):({\sf A}^{\prime}, {\sf H}^{\prime})\rightarrow ({\sf A}^{\prime\prime }, {\sf H}^{\prime\prime})$ is a morphism of matched pairs, the pair $(\Gamma^{\prime}\circ \Gamma, \Omega^{\prime}\circ \Omega)$ is a morphisms of matched pairs. Therefore, matched pairs of quasigropoids form a category that we will denote by  {\sf Mp(QGPD)}. Clearly, if we work with matched pairs of quasigroupoids with the same base $X$ we have a subcategory of  {\sf Mp(QGPD)} denoted by 
{\sf Mp$_{X}$(QGPD)}.
}
	
\end{definition}

\section{Weak Hopf  quasigroups and quasigroupoids}

Throughout this section, ${\mathbb K}$ will be a field. With $id_{V}$ we will denote the identity map for all  $V \in {\mathbb K}$-{\sf Vect} and with $\otimes$ the tensor product over ${\mathbb K}$.

By a unitary magma we will understand a ${\mathbb K}$-vector space $A$ with a unital product. Then we have a multiplication, which is a linear map  $\mu_{A}:A\otimes  A\rightarrow A$, and a unit which can be seen as a linear map $\eta_{A}:{\mathbb K}\rightarrow A$. In this case with $1$ we will denote the unit of $A$, i.e., $1=\eta_{A}(1_{\mathbb K})$, and with $ab$ the image by $\mu_{A}$ of a rank one tensor $a\otimes  b$. Also, by using  the inverse of the scalar multiplication from ${\mathbb K}\otimes   A$ to $A$,  we get the equality $\mu_{A}(1\otimes  a)=a=id_{A}(a)$  (similarly, $\mu_{A}(a\otimes  1)=id_{A}(a)=a$ is obtained by  the inverse of the scalar multiplication from $A\otimes  {\mathbb K}$ to $A$). A map of unitary magmas $f:A\rightarrow B$ is a linear map such that $f(1)=1$ and $f(ab)=f(a)f(b)$.

Let $A$ be a unitary magma. If the product is associative we have an algebra in ${\mathbb K}$-{\sf Vect}. Dually, a coalgebra structure in a vector space $C$ over ${\mathbb K}$ is provided by a coassociative comultiplication, which is a linear map $\delta_{C}:C\rightarrow C\otimes  C$, and a counit linear map $\varepsilon_{C}:C\rightarrow {\mathbb K}$.  For a coalgebra $C$, we will use the Heynenman-Sweedler's notation $\delta_{C}(c)=
c_{(1)}\otimes  c_{(2)}$ with suppressed summation sign. Then $C$ is a coalgebra if the equalities
\begin{align}
	\label{coalg1}
	c_{(1)(1)}\otimes  c_{(1)(2)}\otimes  c_{(2)}&=c_{(1)}\otimes  c_{(2)(1)}\otimes  c_{(2)(2)},\\
	\label{coalg2}
	\varepsilon_{C}(c_{(1)})c_{(2)}&= c=c_{(1)}\varepsilon_{C}(c_{(2)}),  
\end{align}
hold for all $c\in C$. So by the Heyenman-Sweedler's convention we write $c_{(1)}\otimes  c_{(2)}\otimes  c_{(3)}$ for the sums of (\ref{coalg1}). Note that in  (\ref{coalg2}) we identify $\varepsilon_{C}(c_{(1)})\otimes  c_{(2)}$ with the  element of $C$ given by $\varepsilon_{C}(c_{(1)})c_{(2)}$ using the  scalar multiplication from ${\mathbb K}\otimes  C$ to $C$ (similarly, $c_{(1)}\otimes  \varepsilon_{C}(c_{(2)})$ is identified with $c_{(1)} \varepsilon_{C}(c_{(2)})$ by the scalar multiplication from $C\otimes  {\mathbb K}$ to $C$).  A morphism of coalgebras $f:C\rightarrow D$ is a linear map such that  $\varepsilon_{D}(f(c))=\varepsilon_{C}(c)$ and $f(c)_{(1)}\otimes  f(c)_{(2)}=f(c_{(1)})\otimes  f(c_{(2)}).$

If $A$ is a unitary magma and $C$ is a coalgebra, for two linear maps $f,g:C\rightarrow A$, we define the convolution product of $f$ and $g$, denoted by $f\ast g:C\rightarrow A$, as $(f\ast g)(c)=f(c_{(1)}) g(c_{(2)}).$

If $S$ is a set, with ${\mathbb K}[S]$ we will denote the free ${\mathbb K}$-vector space on $S$, i.e.,  
$${\mathbb K}[S]=\displaystyle \bigoplus_{s\in S}{\mathbb K}s.$$

This vector space has a coalgebra structure determined by the coproduct $
\delta_{{\mathbb K}[S]}(s)=s\otimes  s$ and the counit $ \varepsilon_{{\mathbb K}[S]}(s)=1_{\mathbb K}.$

Now we give the definition of weak Hopf quasigroup in ${\mathbb K}$-{\sf Vect}.

\begin{definition}
\label{whq}
{\rm 
A weak Hopf quasigroup in ${\mathbb K}$-{\sf Vect} is a ${\mathbb K}$-vector space $D$ such that it is a unitary magma with product $\mu_{D}(h\otimes  g)=hg$ and unit $1$ and  a coalgebra with coproduct $\delta_{D}$ and counit $\varepsilon_{D}$,  satisfying the following conditions for all $h,k,l \in D$: 
\begin{itemize}
\item[(d1)] $(hk)_{(1)}\otimes (hk)_{(2)}=h_{(1)}k_{(1)}\otimes h_{(2)}k_{(2)}$.
\item[(d2)] $\varepsilon_{D}((hk)l)=\varepsilon_{D}(h(kl))=\varepsilon_{D}(hk_{(1)})\varepsilon_{D}(k_{(2)}l)=\varepsilon_{D}(hk_{(2)})\varepsilon_{D}(k_{(1)}l)$.
\item[(d3)] $1_{(1)}\otimes 1_{(2)}\otimes 1_{(3)}=1_{(1)}\otimes 
1_{(2)}1_{(1^{\prime})}\otimes 1_{(2^{\prime})}=1_{(1)}\otimes 
1_{(1^{\prime})}1_{(2)}\otimes 1_{(2^{\prime})}$.
\item[(d4)] There exists a linear map $\lambda_{D}:D\rightarrow D$ (called the antipode of $D$) such that, if $\Pi_{D}^{L}: D\rightarrow D$ is the ${\mathbb K}$-linear map defined by $\Pi_{D}^{L}=id_{D}\ast \lambda_{D}$ (target morphism) and $\Pi_{D}^{R}: D\rightarrow D$ is the ${\mathbb K}$-linear map defined by $\Pi_{D}^{R}=\lambda_{D}\ast id_{D}$ (source morphism),
\begin{itemize}
\item[(${\rm d4 - 1}$)] $\Pi_{D}^{L}(h)=\varepsilon_{D}(1_{(1)}h)1_{(2)}$.
\item[(${\rm d4-2}$)] $\Pi_{D}^{R}(h)=\varepsilon_{D}(h1_{(2)})1_{(1)}$.
\item[(${\rm d4-3}$)] $\lambda_{D}=\lambda_{D}\ast \Pi_{D}^{L}=\Pi_{D}^{R}\ast \lambda_{D}$.
\item[(${\rm d4-4}$)] $\lambda_{D}(h_{(1)})(h_{(2)}k)=\Pi_{D}^{R}(h)k$.
\item[(${\rm d4-5}$)]$h_{(1)}(\lambda_{D}(h_{(2)})k)=\Pi_{D}^{L}(h)k$.
\item[(${\rm d4-6}$)]$(hk_{(1)})\lambda_{D}(k_{(2)})=h\Pi_{D}^{L}(k)$.
\item[(${\rm d4-7}$)]$(h\lambda_{D}(k_{(1)}))k_{(2)}=h\Pi_{D}^{R}(k)$.	
\end{itemize}
\end{itemize}

Using the composition of morphisms, the previous identities can be expressed as: 
\begin{align*}
{\rm (d1)} & \Leftrightarrow \;   \delta_{D}\co \mu_{D}=(\mu_{D}\ot \mu_{D})\co \delta_{D\ot D},\\
{\rm (d2)} & \Leftrightarrow \;   \varepsilon_{D}\co \mu_{D}\co (\mu_{D}\ot
id_{D})=\varepsilon_{D}\co \mu_{D}\co (id_{D}\ot \mu_{D}) \\ 
\;\;\;\;\;\;\;\;\;\; & \;\;\;\;\;= ((\varepsilon_{D}\co \mu_{D})\ot (\varepsilon_{D}\co
\mu_{D}))\co (id_{D}\ot \delta_{D}\ot id_{D}) \\
\;\;\;\;\;\;\;\;\;\; & \;\;\;\;\;=((\varepsilon_{D}\co \mu_{D})\ot (\varepsilon_{D}\co
\mu_{D}))\co (id_{D}\ot (c_{D,D}\co\delta_{D})\ot id_{D}), \\
{\rm (d3)} & \Leftrightarrow \;  (\delta_{D}\ot id_{D})\co \delta_{D}\co \eta_{D}\\
\;\;\;\;\;\;\;\;\;\; & \;\;\;\;\;=(id_{D}\ot
\mu_{D}\ot id_{D})\co ((\delta_{D}\co \eta_{D}) \ot (\delta_{D}\co
\eta_{D}))\\
\;\;\;\;\;\;\;\;\;\; & \;\;\;\;\;=(id_{D}\ot (\mu_{D}\co c_{D,D})\ot
id_{D})\co ((\delta_{D}\co \eta_{D}) \ot (\delta_{D}\co \eta_{D})),\\
{\rm (d4-1)} & \Leftrightarrow \;  \Pi_{D}^{L}=((\varepsilon_{D}\co
\mu_{D})\ot id_{D})\co (id_{D}\ot c_{D,D})\co ((\delta_{D}\co \eta_{D})\ot
id_{D}),\\
{\rm (d4-2)} &\Leftrightarrow \; \Pi_{D}^{R}=(id_{D}\ot(\varepsilon_{D}\co \mu_{D}))\co (c_{D,D}\ot id_{D})\co (id_{D}\ot (\delta_{D}\co \eta_{D})) ,\\
{\rm (d4-4)} & \Leftrightarrow \; \mu_D\circ (\lambda_D\ot \mu_D)\circ (\delta_D\ot id_{D})=\mu_{D}\co (\Pi_{D}^{R}\ot id_{D}) ,\\
{\rm (d4-5)} & \Leftrightarrow \; \mu_D\circ (id_{D}\ot \mu_D)\circ (id_{D}\ot \lambda_D\ot
id_{D})\circ (\delta_D\ot id_{D})=\mu_{D}\co (\Pi_{D}^{L}\ot id_{D}) ,\\
{\rm (d4-6)} & \Leftrightarrow \; \mu_D\circ(\mu_D\ot \lambda_D)\circ (id_{D}\ot
\delta_D)=\mu_{D}\co (id_{D}\ot \Pi_{D}^{L}) ,\\
{\rm (d4-7)} &  \Leftrightarrow \; \mu_D\circ (\mu_D\ot id_{D})\circ (id_{D}\ot \lambda_D\ot id_{D})\circ (id_{D}\ot \delta_D)=\mu_{D}\co (id_{D}\ot \Pi_{D}^{R})
\end{align*}
where $c_{D,D}$ is the usual twist in ${\mathbb K}$-{\sf vect}, $\eta_{D}$ the unit map, $\mu_{D\otimes D}=(\mu_{D}\otimes \mu_{D})\circ (id_{D}\otimes c_{D,D}\otimes id_{D})$ and $\delta_{D\otimes D}= (id_{D}\otimes c_{D,D}\otimes id_{D})\circ (\delta_{D}\otimes \delta_{D})$.

If  $\varepsilon_D$ and $\delta_D$ are  morphisms of unital magmas, $\Pi_{D}^{L}=\Pi_{D}^{R}=\eta_{D}\ot \varepsilon_{D}$ and, as a consequence, we have  the notion of Hopf quasigroup defined  by Klim and Majid in \cite{KM}. Thus, Hopf quasigroups are examples of weak Hopf quasigroups. 

We will say that $D$ is commutative if $hk=kh$  for all $h,k\in D$, and cocommutative if $k_{(1)}\otimes  k_{(2)}=k_{(2)}\otimes  k_{(1)}$  for all $k \in D$.
}
\end{definition}

\begin{example}
\label{ex-k}
{\rm 
It is well-known that group algebras provide examples of cocommutative Hopf algebras. Also, with finite groupoids it is possible to obtain weak Hopf algebras using the groupoid algebra. In a non-associative setting, J. Klim and S. Majid proved in \cite{KM} that the quasigroup algebra of a quasigroup is a cocommutative Hopf quasigroup. Moreover, as was proved in \cite[Proposition 3.2]{JA21}, any finite quasigroupoid induces  a cocommutative weak Hopf quasigroup, where the base vector space is the quasigroupoid magma. The structure of this kind of weak Hopf quasigroups is the following: Let ${\sf B}=({\sf B}_0,{\sf B}_1)$ be a finite quasigroupoid. The quasigroupoid magma ${\mathbb K}[{\sf B}]$ defined by ${\mathbb K}[{\sf B}]={\mathbb K}[{\sf B}_{1}]$ is a cocommutative weak Hopf quasigroup with ${\mathbb K}$-linear extensions of the unit 
$\displaystyle 1_{{\mathbb K}[{\sf B}]}=\sum_{x\in {\sf B}_0}id_{\sf B}(x)$, product 
$$\mu_{{\mathbb K}[{\sf B}]}(a\otimes b) =\left\{ \begin{array}{l}
a\bullet b \;\; {\rm if} \;\; (a, b)\in {\sf B}_1\;_{s_{\sf B}}\hspace{-0.15cm}\times_{t_{\sf B}} {\sf B}_1\\
\\
0 \;\; {\rm otherwise},
\end{array}\right.
$$
counit $\varepsilon_{{\mathbb K}[{\sf B}]}(b)=1_{\mathbb K}$, coproduct $\delta_{{\mathbb K}[{\sf B}]}(b)=b\otimes b$ and antipode $\lambda_{{\mathbb K}[{\sf B}]}(b)=\lambda_{\sf B}(b)$ on the basis elements.  Note that in this case $\Pi_{{\mathbb K}[{\sf B}]}^{L}(b)=id_{\sf B}(t_{\sf B}(b))$ and $\Pi_{{\mathbb K}[{\sf B}]}^{R}(b)=id_{\sf B}(s_{\sf B}(b))$.
}
\end{example}

\begin{apart}
\label{whqp}
{\rm 
If $D$ is a weak Hopf quasigroup, the  target and source maps are idempotent \cite[Proposition 3.4]{Asian} and, by \cite[Propositions 3.1, 3.2]{Asian}, the following equalities  
$$
	\Pi_{D}^{L}\ast id_{D} =id_{D}\ast  \Pi_{D}^{R}=id_{D},
$$
$$
	\Pi_{D}^{L}\circ \eta_{D}=\eta_{D}=\Pi_{D}^{R}\circ\eta_{D}, \;\;\;\varepsilon_{D}\circ \Pi_{D}^{L}=\varepsilon_{D}=\varepsilon_{D}\circ \Pi_{D}^{R},
$$
hold. Moreover,  the antipode of a  weak Hopf quasigroup $D$ is unique, $\lambda_{D}\circ \eta_{D}=\eta_{D}$,  $\varepsilon_{D}\circ\lambda_{D}=\varepsilon_{D}$ and is antimultiplicative and anticomultiplicative, i.e.,
\begin{align*}
	\lambda_{D}(hg) & =\lambda_{D}(g)\lambda_{D}(h),\\
	\lambda_{D}(h)_{(1)}\otimes  \lambda_{D}(h)_{(2)}& =\lambda_{D}(h_{(2)})\otimes  \lambda_{D}(h_{(1)})
\end{align*}
(see \cite[Theorem 3.19]{Asian}). Moreover, for the morphism target and source we have that 
$$
	\Pi_{D}^{L}=\Pi_{D}^{L}\ast \Pi_{D}^{L}, \;\;\; \Pi_{D}^{R}=\Pi_{D}^{R}\ast \Pi_{D}^{R},
$$
hold (see \cite[Proposition 2.3]{MED}).

On the other hand, if $D_{L}$ is the subspace defined by the image of the target morphism and $h\in D_{L}$, by  \cite[Proposition 2.4]{JPAA} we know that
$$
(hk)l=h(kl), \;\;\;
k(hl)=(kh)l,  \;\;\;
k(lh)=(kl)h,
$$
for all  $k, l\in D$. As a consequence, the unitary magma $D_{L}$ is an algebra in ${\mathbb K}$-{\sf Vect}, where the unit is $1_{D_{L}}=\Pi_{D}^{L}(1)=1$ and $\mu_{D_{L}}=\Pi_{D}^{L}\circ \mu_{D}$.  By \cite[Proposition 2.3]{MED}, we have similar properties for the image of the source morphism, denoted by $D_{R}$. Therefore, if $D_{R}$ is the subspace defined by the image of the source morphism and $h$ is an element of  $D_{R}$, we have that 
$$
(hk)l=h(kl), \;\;\;
k(hl)=(kh)l, \;\;\;	
k(lh)=(kl)h, 
$$
hold for all  $k, l\in D$. Then,   $D_{R}$  is an algebra in ${\mathbb K}$-{\sf Vect} where the unit is $1_{D_{R}}=\Pi_{D}^{R}(1)=1$ and $\mu_{D_{R}}=\Pi_{D}^{R}\circ \mu_{D}$.  For example, in the case of the quasigroupoid magma, the subspaces ${\mathbb K}[{\sf B}]_{L}$ and ${\mathbb K}[{\sf B}]_{R}$ are equal and generated as ${\mathbb K}$-vector spaces by the set $\{id_{\sf B}(x)\; /\; x\in {\sf B}_{0}\}$. 

For a weak Hopf quasigroup $D$, the linear maps $\overline{\Pi}_{D}^{L}, \overline{\Pi}_{D}^{R}:D\rightarrow D$, defined by $$\overline{\Pi}_{D}^{L}(h)=\varepsilon_{D}(1_{(2)}h)1_{(1)}, \;\;\;\;\overline{\Pi}_{D}^{R}(h)=\varepsilon_{D}(h1_{(1)})1_{(2)},$$  are idempotent \cite[Proposition 3.4]{Asian}.  For these maps we have that 
$$
	\overline{\Pi}_{D}^{L}\circ \eta_{D}=\eta_{D}=\overline{\Pi}_{D}^{R}\circ\eta_{D}, \;\;\;\varepsilon_{D}\circ \overline{\Pi}_{D}^{L}=\varepsilon_{D}=\varepsilon_{D}\circ \overline{\Pi}_{D}^{R},
$$
hold. Also, the image of $\overline{\Pi}_{D}^{L}$ is $D_{R}$ and the image of $\overline{\Pi}_{D}^{R}$ is $D_{L}$. Finally, if $D$ is cocommutative, 
$$\Pi_{D}^{L}=\overline{\Pi}_{D}^{L}, \;\;\;\Pi_{D}^{R}=\overline{\Pi}_{D}^{R}.$$

}
\end{apart}

\begin{definition}
\label{morwhq}
{\rm 
Let $D$, $D^{\prime}$ be weak Hopf quasigroups. We will say that a ${\mathbb K}$-linear map $f:D\rightarrow D^{\prime}$ is a morphism of weak Hopf quasigroups if it is a coalgebra morphism such that 
\begin{align}
\label{mkl1}
\Pi^{R}_{D^{\prime}}\circ f & =f\circ \Pi^{R}_{D}, \\
\label{mkl2}
\overline{\Pi}^{L}_{D^{\prime}}\circ f & =f\circ \overline{\Pi}^{L}_{D}, \\
\label{mkl3}
\Pi^{R}_{D^{\prime}}\circ \Pi^{L}_{D^{\prime}}\circ f &=f\circ \Pi^{R}_{D}\circ \Pi^{L}_{D}, \\
\label{mkl4}
f\circ \mu_{D} & =\mu_{D^{\prime}}\circ (f\otimes f)\circ \nabla_{D}, 
\end{align}
hold, where $\nabla_{D}:D\otimes  D\rightarrow D\otimes  D$ is the ${\mathbb K}$-linear idempotent map (see \cite[Definition 3.5]{JA21}) defined by 
$$\nabla_{D}(h\otimes  k)=h_{(1)}\otimes  \Pi_{D}^{R}(h_{(2)})k.$$
		
In the following we will denote by {\sf WHQ} the category whose objects are  weak Hopf quasigroups and whose morphisms are  morphisms of weak Hopf quasigroups  previously defined.
}
\end{definition} 

\begin{example}
\label{exa-mor1}
{\rm 
By \cite[Proposition 3.6]{JA21} we know that if  $\Gamma=(\Gamma_{0},\Gamma_{1}):{\sf A}\rightarrow {\sf A}^{\prime}$ is a morphism of finite quasigroupoids, the linear extension ${\mathbb K}[\Gamma_{1}]: {\mathbb K}[{\sf A}]\rightarrow {\mathbb K}[{\sf A}^{\prime}]$, ${\mathbb K}[\Gamma_{1}](a)= \Gamma_{1}(a)$, is a morphism of weak Hopf quasigroups that we will denote by ${\mathbb K}[\Gamma]$.  Then, if   $({\sf A}, {\sf H})$ is a matched pair of finite quasigroupoids and  $i^{\sf A}: {\sf A}\rightarrow {\sf A}\bowtie {\sf H}$ and $i^{\sf H}: {\sf H}\rightarrow {\sf A}\bowtie {\sf H}$ are the associated monomorphisms, we have that there exists two monomorphisms ${\mathbb K}[i^{\sf A}]: {\mathbb K}[{\sf A}]\rightarrow {\mathbb K}[{\sf A}\bowtie {\sf H}]$ and ${\mathbb K}[i^{\sf H}]: {\mathbb K}[{\sf H}]\rightarrow {\mathbb K}[{\sf A}\bowtie {\sf H}]$ in the category {\sf WHQ}. 
}
\end{example}

\begin{apart}
\label{exmp}
{\rm  Let $({\sf A}, {\sf H})$ be a matched pair of finite quasigroupoids. Let $\varphi_{\sf A}$ be the left action of ${\sf H}$ on ${\sf A}$ and let $\phi_{\sf H}$ be a  right action  of ${\sf A}$ on ${\sf H}$. Then, we can define the following morphisms in ${\mathbb K}$-{\sf Vect} by: 
$$\varphi_{{\mathbb K}[{\sf A}]}: {\mathbb K}[{\sf H}]\otimes {\mathbb K}[{\sf A}]\rightarrow {\mathbb K}[{\sf A}], \;\;\; 
\varphi_{{\mathbb K}[{\sf A}]}(h\otimes a) =
\left\{ \begin{array}{l}
	\varphi_{\sf A} (h,a) \;\; {\rm if} \;\; (h, a)\in {\sf H}_1\;_{s_{\sf H}}\hspace{-0.15cm}\times_{t_{\sf A}} {\sf A}_1
	\\ \\
	0 \;\; {\rm otherwise},
\end{array}\right.
$$
$$\phi_{{\mathbb K}[{\sf H}]}: {\mathbb K}[{\sf H}]\otimes {\mathbb K}[{\sf A}]\rightarrow {\mathbb K}[{\sf H}], \;\;\; 
\phi_{{\mathbb K}[{\sf H}]}(h\otimes a) =
\left\{ \begin{array}{l}
	\phi_{\sf H} (h,a) \;\; {\rm if} \;\; (h, a)\in {\sf H}_1\;_{s_{\sf H}}\hspace{-0.15cm}\times_{t_{\sf A}} {\sf A}_1
	\\ \\
	0 \;\; {\rm otherwise}. 
\end{array}\right.
$$

Note that, by (c3) of Definition \ref{action}, we have that 
$$\varphi_{{\mathbb K}[{\sf H}]}(1_{{\mathbb K}[{\sf A}]}\otimes a) = \sum_{x\in {\sf A}_0}\varphi_{{\mathbb K}[{\sf A}]}(id_{\sf H}(x)\otimes a)=\varphi_{\sf A}(id_{\sf H}(t_{\sf A}(a))\otimes a)=a$$
and, as a consequence, $\varphi_{{\mathbb K}[{\sf A}]}(1_{{\mathbb K}[{\sf A}]}\otimes id_{{\mathbb K}[{\sf A}]})=id_{{\mathbb K}[{\sf A}]}$ holds.

The equality $\varphi_{{\mathbb K}[{\sf A}]}\circ (id_{{\mathbb K}[{\sf A}]}\otimes \varphi_{{\mathbb K}[{\sf A}]})=\varphi_{{\mathbb K}[{\sf A}]}\circ (\mu_{{\mathbb K}[{\sf H}]}\otimes id_{{\mathbb K}[{\sf A}]})$ also holds, because if $(g,h,a)\notin {\sf H}_1\;_{s_{\sf H}}\hspace{-0.15cm}\times_{t_{\sf H}} {\sf H}_1\;_{s_{\sf H}}\hspace{-0.15cm}\times_{t_{\sf A}} {\sf A}_1 $ the evaluation of the two members of the previous identity on $g\otimes h\otimes a$ is zero; On the other hand if $(g,h,a)\in {\sf H}_1\;_{s_{\sf H}}\hspace{-0.15cm}\times_{t_{\sf H}} {\sf H}_1\;_{s_{\sf H}}\hspace{-0.15cm}\times_{t_{\sf A}} {\sf A}_1 $, we have that, by (c2) of Definition \ref{action}, 
$$\varphi_{{\mathbb K}[{\sf A}]}\circ (id_{{\mathbb K}[{\sf A}]}\otimes \varphi_{{\mathbb K}[{\sf A}]})(g\otimes h\otimes a)=\varphi_{{\mathbb K}[{\sf A}]}(g\otimes 	\varphi_{\sf A} (h,a))=	\varphi_{\sf A} (g,\varphi_{\sf A} (h,a))$$
$$=\varphi_{\sf A}(g\star h, a)=\varphi_{{\mathbb K}[{\sf A}]}\circ (\mu_{{\mathbb K}[{\sf H}]}\otimes id_{{\mathbb K}[{\sf A}]})(g\otimes h\otimes a).
$$

Therefore, $({\mathbb K}[{\sf A}], \varphi_{{\mathbb K}[{\sf A}]})$ is an example of left module over the unitary magma ${\mathbb K}[{\sf H}]$. Similarly, by (d2) and (d3) of Definition \ref{action}, we obtain that  $({\mathbb K}[{\sf H}], \phi_{{\mathbb K}[{\sf H}]})$ is an example of right module over the unitary magma ${\mathbb K}[{\sf A}]$.

Let $\Phi:{\mathbb K}[{\sf H}]\otimes {\mathbb K}[{\sf A}]\rightarrow {\mathbb K}[{\sf A}]\otimes {\mathbb K}[{\sf H}]$, 
$\nabla_{\Phi}:{\mathbb K}[{\sf A}]\otimes {\mathbb K}[{\sf H}]\rightarrow {\mathbb K}[{\sf A}]\otimes {\mathbb K}[{\sf H}]$
be the morphisms in ${\mathbb K}$-{\sf Vect} defined by 
$$
\Phi=(\varphi_{{\mathbb K}[{\sf A}]}\otimes \phi_{{\mathbb K}[{\sf H}]})\circ (id_{{\mathbb K}[{\sf A}]}\otimes c_{{\mathbb K}[{\sf H}], {\mathbb K}[{\sf A}]}\otimes id_{{\mathbb K}[{\sf H}]})\circ (\delta_{{\mathbb K}[{\sf H}]}\otimes \delta_{{\mathbb K}[{\sf A}]})
$$
and 
$$
\nabla_{\Phi}=(\mu_{{\mathbb K}[{\sf A}]}\otimes id_{{\mathbb K}[{\sf H}]})\circ (id_{{\mathbb K}[{\sf A}]}\otimes (\Phi\circ (id_{{\mathbb K}[{\sf H}]}\otimes 1_{{\mathbb K}[{\sf A}]}))).
$$

Then, 
$$
\Phi (h\otimes a)= \left\{ \begin{array}{l}
 \varphi_{\sf A}(h,a)\otimes \phi_{\sf H}(h,a)\;\; {\rm if} \;\; (h, a)\in {\sf H}_1\;_{s_{\sf H}}\hspace{-0.15cm}\times_{t_{\sf A}} {\sf A}_1
\\ \\
0 \;\; {\rm otherwise},
\end{array}\right.
$$ 
and 
$$
\nabla_{\Phi} (a\otimes h)= \left\{ \begin{array}{l}
	a\otimes h\;\; {\rm if} \;\; (a, h)\in {\sf A}_1\;_{s_{\sf A}}\hspace{-0.15cm}\times_{t_{\sf H}} {\sf H}_1
	\\ \\
	0 \;\; {\rm otherwise}.
\end{array}\right.
$$

Therefore, $\nabla_{\Phi} $ is an idempotent morphism with image
$$Im(\nabla_{\Phi})=\langle\{ a\otimes h\;\;/ \;\; (a, h)\in {\sf A}_1\;_{s_{\sf A}}\hspace{-0.15cm}\times_{t_{\sf H}} {\sf H}_1 \}\rangle.$$

If we denote by ${\mathbb K}[{\sf A}]\bowtie {\mathbb K}[{\sf H}]$ the image of $\nabla_{\Phi}$, in this ${\mathbb K}$-vector space we can define a product by
\begin{equation}
\label{mump}
\mu_{{\mathbb K}[{\sf A}]\bowtie {\mathbb K}[{\sf H}]}=(\mu_{{\mathbb K}[{\sf A}]}\otimes \mu_{{\mathbb K}[{\sf H}]})\circ (id_{{\mathbb K}[{\sf A}]}\otimes \Phi \otimes id_{{\mathbb K}[{\sf H}]}).
\end{equation}

Then, 
$$
\mu_{{\mathbb K}[{\sf A}]\bowtie {\mathbb K}[{\sf H}]} (a\otimes h\otimes b\otimes g)= \left\{ \begin{array}{l}
a\bullet \varphi_{\sf A}(h,b)\otimes \phi_{\sf A}(h,b)\star g \;\; {\rm if} \;\; (h, b)\in {\sf H}_1\;_{s_{\sf H}}\hspace{-0.15cm}\times_{t_{\sf A}} {\sf A}_1
\\ \\
0 \;\; {\rm otherwise}.
\end{array}\right.
$$

Note that this product is well defined because, by (${\rm a}2-2)$ of Definition \ref{quasigroupoid} and (e1) of Definition \ref{mp}, we have that 
$$s_{\sf A}(a\bullet \varphi_{\sf A}(h,b))=t_{\sf H}(\phi_{\sf A}(h,b)\star g).$$
}
\end{apart}

\begin{theorem}
\label{th1mp}
Let $({\sf A}, {\sf H})$ be a matched pair of finite quasigroupoids. The ${\mathbb K}$-vector  space $${\mathbb K}[{\sf A}]\bowtie {\mathbb K}[{\sf H}]$$ defined in the previous lines is a unitary magma with the product defined in (\ref{mump}) and unit 
$$1_{{\mathbb K}[{\sf A}]\bowtie {\mathbb K}[{\sf H}]}= \sum_{x\in {\sf A}_0}id_{\sf A}(x)\otimes id_{\sf H}(x).$$
\end{theorem}

\begin{proof} We will prove that $\mu_{{\mathbb K}[{\sf A}]\bowtie {\mathbb K}[{\sf H}]}(1_{{\mathbb K}[{\sf A}]\bowtie {\mathbb K}[{\sf H}]}\otimes a\otimes h )=a\otimes h$ for all $(a,h)\in {\mathbb K}[{\sf A}]\bowtie {\mathbb K}[{\sf H}]$. The proof of $\mu_{{\mathbb K}[{\sf A}]\bowtie {\mathbb K}[{\sf H}]}(a\otimes h\otimes 1_{{\mathbb K}[{\sf A}]\bowtie {\mathbb K}[{\sf H}]})=a\otimes h$ is similar and we left the details to the reader.
	
Let $(a,h)\in {\mathbb K}[{\sf A}]\bowtie {\mathbb K}[{\sf H}]$. Then, 
\begin{itemize}
\item[ ]$\hspace{0.38cm}\mu_{{\mathbb K}[{\sf A}]\bowtie {\mathbb K}[{\sf H}]}(1_{{\mathbb K}[{\sf A}]\bowtie {\mathbb K}[{\sf H}]}\otimes a\otimes h )$
\item [ ]$= id_{\sf A}  (t_{\sf A}(a)) \bullet  \varphi_{\sf A}( id_{\sf H}  (t_{\sf A}(a)),a)\otimes \phi_{\sf H}( id_{\sf H}  (t_{\sf A}(a)),a)\star h $ {\scriptsize (by the definition of $\mu_{{\mathbb K}[{\sf A}]\bowtie {\mathbb K}[{\sf H}]}$)}
\item [ ]$=  id_{\sf A}  (t_{\sf A}(a)) \bullet  a\otimes  id_{\sf H}  (s_{\sf A}(a))  \star h$ {\scriptsize (by (c3) of Definition \ref{action} and \ref{P-2}))}
\item [ ]$= a\otimes   id_{\sf H}  (t_{\sf H}(h))  \star h$ {\scriptsize (by $s_{\sf A}(a)=t_{\sf H}(h)$ and (${\rm a}2-1$) of Definition \ref{quasigroupoid})}
\item [ ]$=  a\otimes h  $ {\scriptsize (by (${\rm a}2-1$) of Definition \ref{quasigroupoid})}.
\end{itemize}
\end{proof}

\begin{theorem}
\label{th2mp}
Let $({\sf A}, {\sf H})$ be a matched pair of finite quasigroupoids. The unitary magma $${\mathbb K}[{\sf A}]\bowtie {\mathbb K}[{\sf H}]$$ is a cocommutative weak Hopf quasigroup where for all not nulll element $a\otimes h\in {\mathbb K}[{\sf A}]\bowtie {\mathbb K}[{\sf H}]$, the counit is defined by $\varepsilon_{{\mathbb K}[{\sf A}]\bowtie {\mathbb K}[{\sf H}]}(a\otimes h)=1$, the coproduct by $\delta_{{\mathbb K}[{\sf A}]\bowtie {\mathbb K}[{\sf H}]}(a\otimes h)=a\otimes h\otimes a\otimes h$ and the antipode by $$\lambda_{{\mathbb K}[{\sf A}]\bowtie {\mathbb K}[{\sf H}]}(a\otimes h)=\varphi_{\sf A}(\lambda_{\sf H}(h), \lambda_{\sf A}(a))\otimes \phi_{\sf H}(\lambda_{\sf H}(h), \lambda_{\sf A}(a)).$$
\end{theorem}

\begin{proof} In order to simplify the notation, in this proof we will denote the ${\mathbb K}$-vector space ${\mathbb K}[{\sf A}]\bowtie {\mathbb K}[{\sf H}]$ by $D$.   

First note that it is easy to show that $D$ is cocommutative coalgebra. Also, if $a\otimes h\otimes b\otimes g \in D\otimes D$, we have that 
\begin{align*}
\;	&\hspace{0.38cm}((\mu_{D}\ot \mu_{D})\co \delta_{D\ot D})(a\otimes h\otimes b\otimes g)\\
\;	& = \left\{ \begin{array}{l}
a\bullet \varphi_{\sf A}(h,b)\otimes \phi_{\sf A}(h,b)\star g\otimes a\bullet \varphi_{\sf A}(h,b)\otimes \phi_{\sf A}(h,b)\star g	 \;\; {\rm if} \;\; (h, b)\in {\sf H}_1\;_{s_{\sf H}}\hspace{-0.15cm}\times_{t_{\sf A}} {\sf A}_1\\
	\\
	0 \;\; {\rm otherwise}
\end{array}\right.\\
\;	& = (\delta_{D}\circ \mu_{D}) (a\otimes h\otimes b\otimes g)
\end{align*}

Therefore, (d1) of Definition \ref{whq} holds.  The condition (d2) of Definition \ref{whq}  holds because 
\begin{align*}
\;	& \hspace{0.42cm} (\varepsilon_{D}\co \mu_{D}\co (\mu_{D}\ot id_{D}))(a\otimes h\otimes b\otimes g\otimes c\otimes l)\\
\; & = (\varepsilon_{D}\co \mu_{D}\co (id_{D}\ot \mu_{D}))(a\otimes h\otimes b\otimes g\otimes c\otimes l)\\
\;	& = (((\varepsilon_{D}\co \mu_{D})\ot (\varepsilon_{D}\co
\mu_{D}))\co (id_{D}\ot \delta_{D}\ot id_{D}))(a\otimes h\otimes b\otimes g\otimes c\otimes l)\\
\;	& =(((\varepsilon_{D}\co \mu_{D})\ot (\varepsilon_{D}\co
\mu_{D}))\co (id_{D}\ot (c_{D,D}\co\delta_{D})\ot id_{D}))(a\otimes h\otimes b\otimes g\otimes c\otimes l)\\
\;	& = \left\{ \begin{array}{l}
1	 \;\; {\rm if} \;\; (h, b)\in {\sf H}_1\;_{s_{\sf H}}\hspace{-0.15cm}\times_{t_{\sf A}} {\sf A}_1\;{\rm and}\;  (g, c)\in {\sf H}_1\;_{s_{\sf H}}\hspace{-0.15cm}\times_{t_{\sf A}} {\sf A}_1\\
\\
0 \;\; {\rm otherwise}
\end{array}\right.
\end{align*}
and (d3) of Definition \ref{whq} follows by
\begin{align*}
\;	&  \hspace{0.42cm}  ((\delta_{D}\ot id_{D})\co \delta_{D}) (1_{D})\\
\; & =((id_{D}\ot
\mu_{D}\ot id_{D})\co (\delta_{D}(1_{D})) \ot \delta_{D}(1_{D}))\\
\; & =(id_{D}\ot (\mu_{D}\co c_{D,D})\ot
id_{D})\co (\delta_{D}(1_{D}) \ot \delta_{D} (1_{D})),\\
\; & =  \sum_{x\in {\sf A}_0} id_{\sf A}(x)\otimes id_{\sf H}(x)\otimes id_{\sf A}(x)\otimes id_{\sf H}(x)\otimes id_{\sf A}(x)\otimes id_{\sf H}(x).
\end{align*}

The antipode $\lambda_{D}$ is well defined because by (e1) of Definition \ref{mp}, we have that $$s_{\sf A}(\varphi_{\sf A}(\lambda_{\sf H}(h), \lambda_{\sf A}(a)))=t_{\sf H}(\phi_{\sf H}(\lambda_{\sf H}(h), \lambda_{\sf A}(a))).$$ 

Then the target map is 
\begin{itemize}
\item[ ]$\hspace{0.38cm}\Pi_{D}^{L}(a\otimes h)$ 
\item [ ]$= a\bullet \varphi_{\sf A}(h,  \varphi_{\sf A}(\lambda_{\sf H}(h), \lambda_{\sf A}(a)))\otimes  \phi_{\sf H}(h,  \varphi_{\sf A}(\lambda_{\sf H}(h), \lambda_{\sf A}(a)))\star \phi_{\sf H}(\lambda_{\sf H}(h), \lambda_{\sf A}(a))$ {\scriptsize (by definition of $\Pi_{D}^{L}$)}
\item [ ]$=a\bullet \varphi_{\sf A}(h\star \lambda_{\sf H}(h), \lambda_{\sf A}(a))\otimes  \phi_{\sf H}(\lambda_{\sf H}(\lambda_{\sf H}(h)),  \varphi_{\sf A}(\lambda_{\sf H}(h), \lambda_{\sf A}(a)))\star \phi_{\sf H}(\lambda_{\sf H}(h), \lambda_{\sf A}(a))  $ {\scriptsize (by (c2) of }
\item[ ]$\hspace{0.38cm}$ {\scriptsize  Definition \ref{action} and (\ref{E-5}))}
\item [ ]$=a\bullet \varphi_{\sf A}(id_{\sf H}(t_{\sf H}(h)), \lambda_{\sf A}(a)) \otimes \lambda_{H}(\phi_{\sf H}(\lambda_{\sf H}(h), \lambda_{\sf A}(a)))\star \phi_{\sf H}(\lambda_{\sf H}(h), \lambda_{\sf A}(a))    $ {\scriptsize (by (\ref{E-4}) and (\ref{P-4}))}
\item [ ]$= a\bullet \varphi_{\sf A}(id_{\sf H}(s_{\sf A}(a)), \lambda_{\sf A}(a)) \otimes id_{\sf H}(s_{\sf H}(\phi_{\sf H}(\lambda_{\sf H}(h), \lambda_{\sf A}(a))))$ {\scriptsize (by $s_{\sf A}(a)=t_{\sf H}(h)$ and (\ref{E-3}))}
\item [ ]$=a\bullet \varphi_{\sf A}(id_{\sf H}(t_{\sf A}(\lambda_{\sf A}(a))), \lambda_{\sf A}(a)) \otimes id_{\sf H}(s_{\sf A}(\lambda_{\sf A}(a)))  $ {\scriptsize (by (\ref{E-3}) and (d1) of Definition \ref{action})}
\item [ ]$=a\bullet \lambda_{\sf A}(a) \otimes  id_{\sf H}(t_{\sf A}(a))$ {\scriptsize (by  (c3) of Definition \ref{action} and (\ref{E-1}))}
\item [ ]$=  id_{\sf A}(t_{\sf A}(a)) \otimes  id_{\sf H}(t_{\sf A}(a))$ {\scriptsize (by (\ref{E-4}))}
\end{itemize}
and,  similarly, the expression for the source map is
$$\Pi_{D}^{R}(a\otimes h)= id_{\sf A}(s_{\sf H}(h)) \otimes  id_{\sf H}(s_{\sf H}(h)).$$

As a consequence, by a reasoning similar to that developed in the previous calculations, we have that (${\rm d}4-1$) and (${\rm d}4-2$)  of Definition \ref{whq} hold.  For example,  (${\rm d}4-1$) follows by 
\begin{align*}
\;	&  \hspace{0.42cm} (((\varepsilon_{D}\co
\mu_{D})\ot id_{D})\co (id_{D}\ot c_{D,D}))(\delta_{D}(1_{D})\ot
a\otimes h)\\
\;	& =(\varepsilon_{D}\co
\mu_{D})(id_{\sf A}(t_{\sf A}(a))\bullet \varphi_{\sf A}(id_{\sf H}(t_{\sf A}(a)), a)\otimes \phi_{\sf H}(id_{\sf H}(t_{\sf A}(a)), a)\star h)\otimes  id_{\sf A}(t_{\sf A}(a)) \otimes  id_{\sf H}(t_{\sf A}(a))\\
\;	& =(\varepsilon_{D}\co
\mu_{D})(a\otimes h)\otimes  id_{\sf A}(t_{\sf A}(a)) \otimes  id_{\sf H}(t_{\sf A}(a))\\
\;	& =id_{\sf A}(t_{\sf A}(a)) \otimes  id_{\sf H}(t_{\sf A}(a))\\
\;	& =\Pi_{D}^{L}(a\otimes h).
\end{align*}

Also, we can prove that $\lambda_{D}\ast \Pi_{D}^{L}=\lambda_{D}$. Indeed: 
\begin{itemize}
\item[ ]$\hspace{0.38cm}(\lambda_{D}\ast \Pi_{D}^{L})(a\otimes h)$ 
\item [ ]$=  \varphi_{\sf A}(\lambda_{\sf H}(h), \lambda_{\sf A}(a))\bullet    \varphi_{\sf A}( \phi_{\sf H}(\lambda_{\sf H}(h), \lambda_{\sf A}(a)), id_{\sf A}(t_{\sf A}(a)) )\otimes  \phi_{\sf H}( \phi_{\sf H}(\lambda_{\sf H}(h), \lambda_{\sf A}(a)), id_{\sf A}(t_{\sf A}(a)) ) $ 
\item[ ]$\hspace{0.38cm}\star  id_{\sf H}(t_{\sf A}(a))$ {\scriptsize (by the definitions of $\lambda_{D}$, $\Pi_{D}^{L}$ and the convolution product)}
	
\item [ ]$=  \varphi_{\sf A}(\lambda_{\sf H}(h), \lambda_{\sf A}(a))\bullet    \varphi_{\sf A}( \phi_{\sf H}(\lambda_{\sf H}(h), \lambda_{\sf A}(a)), id_{\sf A}(s_{\sf H}( \phi_{\sf H}(\lambda_{\sf H}(h), \lambda_{\sf A}(a)))) )  $ 
\item[ ]$\hspace{0.38cm}\otimes  \phi_{\sf H}( \phi_{\sf H}(\lambda_{\sf H}(h), \lambda_{\sf A}(a)), id_{\sf A}(s_{\sf H}( \phi_{\sf H}(\lambda_{\sf H}(h), \lambda_{\sf A}(a))) ) \star  id_{\sf H}(s_{\sf H}( \phi_{\sf H}(\lambda_{\sf H}(h), \lambda_{\sf A}(a))))$ 
\item[ ]$\hspace{0.38cm}$ {\scriptsize (by (\ref{E-1}) and (d1) of Definition \ref{action})}
	
\item [ ]$=  \varphi_{\sf A}(\lambda_{\sf H}(h), \lambda_{\sf A}(a))\bullet    id_{\sf A}(t_{\sf H}( \phi_{\sf H}(\lambda_{\sf H}(h), \lambda_{\sf A}(a)))) )$
\item[ ]$\hspace{0.38cm}\otimes   \phi_{\sf H}(\lambda_{\sf H}(h), \lambda_{\sf A}(a))\star  id_{\sf H}(s_{\sf H}( \phi_{\sf H}(\lambda_{\sf H}(h), \lambda_{\sf A}(a))))$ {\scriptsize (by (\ref{P-1}) and (d3) of Definition \ref{action})}
	
\item [ ]$=  \varphi_{\sf A}(\lambda_{\sf H}(h), \lambda_{\sf A}(a))\bullet    id_{\sf A}(s_{\sf A}( \varphi_{\sf A}(\lambda_{\sf H}(h), \lambda_{\sf A}(a)))) )\otimes   \phi_{\sf H}(\lambda_{\sf H}(h), \lambda_{\sf A}(a))$ {\scriptsize (by (e1) of Definition \ref{mp} }
\item[ ]$\hspace{0.38cm}$ {\scriptsize and (${\rm a}2-1$) of Definition \ref{quasigroupoid})}
	
\item [ ]$=\lambda_{D}(a\otimes h)${\scriptsize (by  (${\rm a}2-1$) of Definition \ref{quasigroupoid})}.
\end{itemize}

Similarly, $\Pi_{D}^{R}\ast \lambda_{D}=\lambda_{D}$ and therefore (${\rm d}4-3$)  of Definition \ref{whq} holds.

To end the proof, we prove that equalities (${\rm d}4-4$)  and  (${\rm d}4-5$) of Definition \ref{whq}. The proof for (${\rm d}4-6$)  and  (${\rm d}4-7$) follow a similar pattern and we leave the details to the reader.  

Let $a\otimes h$ and  $b\otimes g$ be in $D$. Then, by the definitions of $\mu_{D}$ and $\lambda_{D}$ we have that 
$$(\mu_D\circ (\lambda_D\ot \mu_D)\circ (\delta_D\ot id_{D}))(a\otimes h\otimes b\otimes g)$$
$$= \left\{ \begin{array}{l}
 \left\{ \begin{array}{l} 	\varphi_{\sf A}(\lambda_{\sf H}(h), \lambda_{\sf A}(a))\bullet \varphi_{\sf A}(\phi_{\sf H}(\lambda_{\sf H}(h), \lambda_{\sf A}(a)), a\bullet \varphi_{\sf A}(h, b))
\\
\otimes\;\; \phi_{\sf H}(\phi_{\sf H}(\lambda_{\sf H}(h), \lambda_{\sf A}(a)), a\bullet \varphi_{\sf A}(h, b))\star (\phi_{\sf H}(h, b)\star g)) \;\;\;\;\; {\rm if} \;\; (h, b)\in {\sf H}_1\;_{s_{\sf H}}\hspace{-0.15cm}\times_{t_{\sf A}} {\sf A}_1,\\ \end{array}\right.
\\
\\
0 \;\; {\rm otherwise.}
\end{array}\right.
$$

Then, 
\begin{itemize}
\item[ ]$\hspace{0.38cm}	\varphi_{\sf A}(\lambda_{\sf H}(h), \lambda_{\sf A}(a))\bullet \varphi_{\sf A}(\phi_{\sf H}(\lambda_{\sf H}(h), \lambda_{\sf A}(a)), a\bullet \varphi_{\sf A}(h, b))$ 
\item [ ]$= \varphi_{\sf A}(\lambda_{\sf H}(h), \lambda_{\sf A}(a))\bullet  (\varphi_{\sf A}(\phi_{\sf H}(\lambda_{\sf H}(h), \lambda_{\sf A}(a)), a)\bullet  (\varphi_{\sf A}(\phi_{\sf H}(\phi_{\sf H}(\lambda_{\sf H}(h), \lambda_{\sf A}(a)), a), \varphi_{\sf A}(h,b)))) $ {\scriptsize (by }
\item[ ]$\hspace{0.38cm}$  {\scriptsize  (e2) of Definition \ref{mp})}
\item [ ]$=  \varphi_{\sf A}(\lambda_{\sf H}(h), \lambda_{\sf A}(a))\bullet (\varphi_{\sf A}(\phi_{\sf H}(\lambda_{\sf H}(h), \lambda_{\sf A}(a)), \lambda_{\sf A}(\lambda_{\sf A}(a)))\bullet   \varphi_{\sf A}(\phi_{\sf H}(\lambda_{\sf H}(h), \lambda_{\sf A}(a)\bullet a), \varphi_{\sf A}(h,b))) $ 
\item[ ]$\hspace{0.38cm}$ {\scriptsize (by (\ref{E-5}) and (d2) of Definition \ref{action}))}
 \item[ ]$=\varphi_{\sf A}(\lambda_{\sf H}(h), \lambda_{\sf A}(a))\bullet (\lambda_{A}(\varphi_{\sf A}(\lambda_{\sf H}(h), \lambda_{\sf A}(a)))\bullet   \varphi_{\sf A}(\phi_{\sf H}(\lambda_{\sf H}(h), id_{\sf A}(s_{\sf A}(a)), \varphi_{\sf A}(h,b)))) $ 
\item[ ]$\hspace{0.38cm}$ {\scriptsize (by  (\ref{P-3}) and (\ref{E-3}))}
\item [ ]$= \varphi_{\sf A}(\phi_{\sf H}(\lambda_{\sf H}(h), id_{\sf A}(s_{\sf A}(a)))), \varphi_{\sf A}(h,b)))  $ {\scriptsize (by (${\rm a}2-3$) of Definition \ref{quasigroupoid})}
\item [ ]$= \varphi_{\sf A}(\phi_{\sf H}(\lambda_{\sf H}(h), id_{\sf A}(t_{\sf H}(h))), \varphi_{\sf A}(h,b))  $ {\scriptsize (by $s_{\sf A}(a)=t_{\sf H}(h)$)}
\item [ ]$=\varphi_{\sf A}(\phi_{\sf H}(\lambda_{\sf H}(h), id_{\sf A}(s_{\sf H}(\lambda_{\sf H}(h)))), \varphi_{\sf A}(h,b)) $ {\scriptsize (by (\ref{E-1}))}
\item [ ]$= \varphi_{\sf A}(\lambda_{\sf H}(h), \varphi_{\sf A}(h,b)))$ {\scriptsize (by (d3) of Definition \ref{action})}
\item [ ]$=\varphi_{\sf A}(\lambda_{\sf H}(h)\star h,b) $ {\scriptsize (by (c2) of Definition \ref{action})}
\item [ ]$=\varphi_{\sf A}(id_{\sf H}(s_{\sf H}(h)),b) $ {\scriptsize (by (\ref{E-3}))}
\item [ ]$= \varphi_{\sf A}(id_{\sf H}(t_{\sf A}(b)),b) $ {\scriptsize (by $t_{\sf A}(b)=s_{\sf H}(h)$)}
\item [ ]$= b$ {\scriptsize (by (c3) of Definition \ref{action})}
\end{itemize}
and 
\begin{itemize}
	\item[ ]$\hspace{0.38cm}\phi_{\sf H}(\phi_{\sf H}(\lambda_{\sf H}(h), \lambda_{\sf A}(a)), a\bullet \varphi_{\sf A}(h, b))\star (\phi_{\sf H}(h, b)\star g))$ 
	\item [ ]$=\phi_{\sf H}(\lambda_{\sf H}(h), \lambda_{\sf A}(a)\bullet (a\bullet \varphi_{\sf A}(h, b)) )\star (\phi_{\sf H}(h, b)\star g))$ {\scriptsize (by  (d2) of Definition \ref{action})}
	\item[ ]$=\phi_{\sf H}(\lambda_{\sf H}(h), \varphi_{\sf A}(h, b) )\star (\phi_{\sf H}(h, b)\star g))$  {\scriptsize (by $({\rm a}2-3$) of Definition \ref{quasigroupoid})}
	\item [ ]$=\lambda_{\sf H}( \phi_{\sf H}(h, b)) \star (\phi_{\sf H}(h, b)\star g)) $  {\scriptsize (by (\ref{P-4}))}
	\item[ ]$=g $  {\scriptsize (by ${\rm a}2-3$) of Definition \ref{quasigroupoid})}.
\end{itemize}

Therefore, 
$$(\mu_D\circ (\lambda_D\ot \mu_D)\circ (\delta_D\ot id_{D}))(a\otimes h\otimes b\otimes g)$$
$$= \left\{ \begin{array}{l}
b\otimes g \;\;\;{\rm if} \;\; (h, b)\in {\sf H}_1\;_{s_{\sf H}}\hspace{-0.15cm}\times_{t_{\sf A}} {\sf A}_1,\\ 
	\\
	0 \;\; {\rm otherwise.}
\end{array}\right.
$$

On the other hand, using similar arguments we obtain
\begin{align*}
\;	& \;\;\;\;\;\;(\mu_D\circ ( \Pi_{D}^{R}\otimes id_{D}))(a\otimes h\otimes b\otimes g)\\
\;	& = \left\{ \begin{array}{l}
id_{\sf A}(s_{\sf H}(h))\bullet \varphi_{\sf A}(id_{\sf H}(s_{\sf H}(h)), b)\otimes \varphi_{\sf A}(id_{\sf H}(s_{\sf H}(h)), b)\star g	 \;\;\; {\rm if} \;\; (h, b)\in {\sf H}_1\;_{s_{\sf H}}\hspace{-0.15cm}\times_{t_{\sf A}} {\sf A}_1,\\
\\	
0 \;\; {\rm otherwise.}
\end{array}\right.\\
\;	& = \left\{ \begin{array}{l}
id_{\sf A}(t_{\sf A}(b))\bullet \varphi_{\sf A}(id_{\sf H}(t_{\sf A}(b)), b)\otimes \phi_{\sf H}(id_{\sf H}(t_{\sf A}(b)), b)\star g	 \;\;\; {\rm if} \;\; (h, b)\in {\sf H}_1\;_{s_{\sf H}}\hspace{-0.15cm}\times_{t_{\sf A}} {\sf A}_1,\\ 
\\
0 \;\; {\rm otherwise.}
\end{array}\right.\\
\;	& = \left\{ \begin{array}{l}
id_{\sf A}(t_{\sf A}(b))\bullet b\otimes id_{\sf A}(s_{\sf A}(b))\star g	 \;\;\; {\rm if} \;\; (h, b)\in {\sf H}_1\;_{s_{\sf H}}\hspace{-0.15cm}\times_{t_{\sf A}} {\sf A}_1,\\ 
\\
0 \;\; {\rm otherwise.}
\end{array}\right.\\
\;	& = \left\{ \begin{array}{l}
b\otimes g \;\;\;{\rm if} \;\; (h, b)\in {\sf H}_1\;_{s_{\sf H}}\hspace{-0.15cm}\times_{t_{\sf A}} {\sf A}_1,\\ 
\\
0 \;\; {\rm otherwise.}
\end{array}\right.
\end{align*}

Thus, (${\rm d}4-4$)  of Definition \ref{whq} holds. 

Finally, 
$$(\mu_D\circ (id_{D}\ot \mu_D)\circ (id_{D}\ot \lambda_D\ot
id_{D})\circ (\delta_D\ot id_{D}))(a\otimes h\otimes b\otimes g)$$
$$= \left\{ \begin{array}{l}
\left\{ \begin{array}{l} 	
a\bullet \varphi_{\sf A}(h, \varphi_{\sf A}(\lambda_{\sf H}(h), \lambda_{\sf A}(a))\bullet \varphi_{\sf A}(\phi_{\sf H}(\lambda_{\sf H}(h), \lambda_{\sf A}(a)), b))  \\
\otimes \phi_{\sf H}(h, \varphi_{\sf A}(\lambda_{\sf H}(h), \lambda_{\sf A}(a))\bullet \varphi_{\sf A}(\phi_{\sf H}(\lambda_{\sf H}(h), \lambda_{\sf A}(a)), b))\star (\phi_{\sf H}(\phi_{\sf H}(\lambda_{\sf H}(h), \lambda_{\sf A}(a)), b)\star g )\\
 {\rm if} \;\; (a, b)\in {\sf A}_1\;_{t_{\sf A}}\hspace{-0.15cm}\times_{t_{\sf A}} {\sf A}_1,\\ \end{array}\right.
\\
\\
0 \;\; {\rm otherwise,}
\end{array}\right.
$$
where 
\begin{itemize}
\item[ ]$\hspace{0.38cm}a\bullet \varphi_{\sf A}(h, \varphi_{\sf A}(\lambda_{\sf H}(h), \lambda_{\sf A}(a))\bullet \varphi_{\sf A}(\phi_{\sf H}(\lambda_{\sf H}(h), \lambda_{\sf A}(a)), b))  $ 
\item [ ]$=a\bullet \varphi_{\sf A}(h, \varphi_{\sf A}(\lambda_{\sf H}(h), \lambda_{\sf A}(a)\bullet b)) $ {\scriptsize (by  (e2) of Definition \ref{mp})}
\item[ ]$= a\bullet \varphi_{\sf A}(h\star \lambda_{\sf H}(h), \lambda_{\sf A}(a)\bullet b)  $ {\scriptsize (by (d2) of Definition \ref{action})}
\item[ ]$= a\bullet \varphi_{\sf A}(id_{\sf H}(t_{\sf H}(h)), \lambda_{\sf A}(a)\bullet b)  $  {\scriptsize (by (\ref{E-4}))}
\item[ ]$= a\bullet \varphi_{\sf A}(id_{\sf H}(s_{\sf A}(a)), \lambda_{\sf A}(a)\bullet b)  $ {\scriptsize  (by $t_{\sf H}(h)=s_{\sf A}(a)$)}
\item[ ]$=a\bullet \varphi_{\sf A}(id_{\sf H}(t_{\sf A}(\lambda_{\sf A}(a)\bullet b)), \lambda_{\sf A}(a)\bullet b)$  {\scriptsize (by (\ref{E-4}) and $({\rm a}2-2$) of Definition \ref{quasigroupoid})}
\item[ ]$= a\bullet  (\lambda_{\sf A}(a)\bullet b) $  {\scriptsize (by (c3) of Definition \ref{action})}
\item[ ]$= \lambda_{\sf A}(\lambda_{\sf A}(a))\bullet  (\lambda_{\sf A}(a)\bullet b) $ {\scriptsize  (by (\ref{E-4}))}
\item[ ]$=b$ {\scriptsize (by $({\rm a}2-3$) of Definition \ref{quasigroupoid})}
\end{itemize}
and 
\begin{itemize}
	\item[ ]$\hspace{0.38cm} \phi_{\sf H}(h, \varphi_{\sf A}(\lambda_{\sf H}(h), \lambda_{\sf A}(a))\bullet \varphi_{\sf A}(\phi_{\sf H}(\lambda_{\sf H}(h), \lambda_{\sf A}(a)), b))\star (\phi_{\sf H}(\phi_{\sf H}(\lambda_{\sf H}(h), \lambda_{\sf A}(a)), b)\star g)  $ 
	\item[ ]$= \phi_{\sf H}(h, \varphi_{\sf A}(\lambda_{\sf H}(h), \lambda_{\sf A}(a)\bullet b))\star (\phi_{\sf H}(\lambda_{\sf H}(h), \lambda_{\sf A}(a)\bullet b)\star g) $ {\scriptsize (by (e2) of Definition \ref{mp} and (d2) of}
	\item[ ]$\hspace{0.38cm}$ {\scriptsize Definition \ref{action})}
	\item[ ]$= \phi_{\sf H}(\lambda_{\sf H}(\lambda_{\sf H}(h)), \varphi_{\sf A}(\lambda_{\sf H}(h), \lambda_{\sf A}(a)\bullet b))\star (\phi_{\sf H}(\lambda_{\sf H}(h), \lambda_{\sf A}(a)\bullet b)\star g)  $  {\scriptsize (by (\ref{E-5})}
	\item[ ]$= \lambda_{\sf H}(\phi_{\sf H}(\lambda_{\sf H}(h), \lambda_{\sf A}(a)\bullet b))\star (\phi_{\sf H}(\lambda_{\sf H}(h), \lambda_{\sf A}(a)\bullet b)\star g)$  {\scriptsize (by (\ref{P-4})}
	\item[ ]$=g $  {\scriptsize (by $({\rm a}2-3$) of Definition \ref{quasigroupoid})}.
\end{itemize}

Thus 
$$(\mu_D\circ (id_{D}\ot \mu_D)\circ (id_{D}\ot \lambda_D\ot
id_{D})\circ (\delta_D\ot id_{D}))(a\otimes h\otimes b\otimes g)$$
$$= \left\{ \begin{array}{l} 	
b\otimes g \;\;{\rm if} \;\; (a, b)\in {\sf A}_1\;_{t_{\sf A}}\hspace{-0.15cm}\times_{t_{\sf A}} {\sf A}_1,\\
\\
0 \;\; {\rm otherwise}
\end{array}\right.
$$
and taking into account that
\begin{align*}
\;	& \;\;\;\;\;\;(\mu_D\circ ( \Pi_{D}^{L}\otimes id_{D}))(a\otimes h\otimes b\otimes g)\\
\;	& = \left\{ \begin{array}{l}
id_{\sf A}(t_{\sf A}(a))\bullet \varphi_{\sf A}(id_{\sf H}(t_{\sf A}(a)), b)\otimes \phi_{\sf H}(id_{\sf H}(t_{\sf A}(a)), b)\star g	 \;\;\; {\rm if} \;\; (a, b)\in {\sf A}_1\;_{t_{\sf A}}\hspace{-0.15cm}\times_{t_{\sf A}} {\sf A}_1,\\
\\	
0 \;\; {\rm otherwise.}
\end{array}\right.\\
\;	& = \left\{ \begin{array}{l}
id_{\sf A}(t_{\sf A}(b))\bullet \varphi_{\sf A}(id_{\sf H}(t_{\sf A}(b)), b)\otimes \phi_{\sf H}(id_{\sf H}(t_{\sf A}(b)), b)\star g	 \;\;\; {\rm if} \;\; (a, b)\in {\sf A}_1\;_{t_{\sf A}}\hspace{-0.15cm}\times_{t_{\sf A}} {\sf A}_1,\\ 
\\
0 \;\; {\rm otherwise.}
\end{array}\right.\\
\;	& = \left\{ \begin{array}{l}
id_{\sf A}(t_{\sf A}(b))\bullet b\otimes id_{\sf H}(s_{\sf A}(b))\star g	 \;\;\; {\rm if} \;\; (h, b)\in {\sf A}_1\;_{t_{\sf A}}\hspace{-0.15cm}\times_{t_{\sf A}} {\sf A}_1,\\ 
\\
0 \;\; {\rm otherwise.}
\end{array}\right.\\
\;	& = \left\{ \begin{array}{l}
b\otimes g  \;\;\; {\rm if} \;\; (h, b)\in {\sf A}_1\;_{t_{\sf A}}\hspace{-0.15cm}\times_{t_{\sf A}} {\sf A}_1,\\ 
\\
0 \;\; {\rm otherwise,}
\end{array}\right.
\end{align*}
we can ensure that (${\rm d}4-5$) of Definition \ref{whq} holds
\end{proof}

\begin{theorem}
Let $({\sf A}, {\sf H})$ be a matched pair of finite quasigroupoids. The cocommutative  weak Hopf quasigroups ${\mathbb K}[{\sf A}\bowtie {\sf H}]$ and ${\mathbb K}[{\sf A}]\bowtie {\mathbb K}[{\sf H}]$ are isomorphic in {\sf WHQ}.
\end{theorem}

\begin{proof} Let $({\sf A}, {\sf H})$ be a matched pair of finite quasigroupoids. First note that the weak Hopf quasigroups ${\mathbb K}[{\sf A}\bowtie {\sf H}]$ and ${\mathbb K}[{\sf A}]\bowtie {\mathbb K}[{\sf H}]$ have the same dimension as ${\mathbb K}$-vector spaces. Also, they are cocommutative and then $\Pi_{B}^{L}=\overline{\Pi}_{B}^{L}$ and $\Pi_{B}^{R}=\overline{\Pi}_{B}^{R}$, for $B$ equal to ${\mathbb K}[{\sf A}\bowtie {\sf H}]$ or $B$ equal to ${\mathbb K}[{\sf A}]\bowtie {\mathbb K}[{\sf H}]$. 
	
Define the ${\mathbb K}$-linear map $f:{\mathbb K}[{\sf A}\bowtie {\sf H}]\rightarrow {\mathbb K}[{\sf A}]\bowtie {\mathbb K}[{\sf H}]$ as the one such that 
$$f(a,g)=a\ot g$$
for all $(a,g)\in {\sf A}_1\;_{s_{\sf A}}\hspace{-0.15cm}\times_{t_{\sf H}} {\sf H}_1$. Then, $f$ is surjective and then it is bijective. Moreover, $f$ is a morphism {\sf WHQ}. Indeed, to prove this assertion we must show that conditions (\ref{mkl1}), (\ref{mkl2}), (\ref{mkl3}) and (\ref{mkl4}) of Definition \ref{morwhq} hold. Note that, by the cocommutative condition,  we only need to prove (\ref{mkl1}), (\ref{mkl2}) for $\Pi^{L}_{X}$ with $X\in\{{\mathbb K}[{\sf A}\bowtie {\sf H}],  {\mathbb K}[{\sf A}]\bowtie {\mathbb K}[{\sf H}]\}$ and (\ref{mkl4}). Under these conditions,  (\ref{mkl3}) follows from (\ref{mkl1}) and (\ref{mkl2}).

Let  $(a,g)$ be in $ {\sf A}_1\;_{s_{\sf A}}\hspace{-0.15cm}\times_{t_{\sf H}} {\sf H}_1$. Then, 
$$
(\Pi^{R}_{{\mathbb K}[{\sf A}]\bowtie {\mathbb K}[{\sf H}]}\circ f)(a,g)=\Pi^{R}_{{\mathbb K}[{\sf A}]\bowtie {\mathbb K}[{\sf H}]} (a\otimes g)=id_{\sf A}(s_{\sf H}(g))\otimes id_{\sf H}(s_{\sf H}(g))=f(id_{\sf A}(s_{\sf H}(g)), id_{\sf H}(s_{\sf H}(g)))$$
$$=f(id_{{\sf A}\bowtie {\sf H}}(s_{\sf H}(g)))=f(id_{{\sf A}\bowtie {\sf H}}(s_{{\sf A}\bowtie {\sf H}}(a,g)))=
(f\circ \Pi^{R}_{{\mathbb K}[{\sf A}\bowtie {\sf H}]})(a,g)
$$
and 
$$
(\Pi^{L}_{{\mathbb K}[{\sf A}]\bowtie {\mathbb K}[{\sf H}]}\circ f)(a,g)=\Pi^{L}_{{\mathbb K}[{\sf A}]\bowtie {\mathbb K}[{\sf H}]} (a\otimes g)=id_{\sf A}(t_{\sf A}(a))\otimes id_{\sf H}(t_{\sf A}(a))=f(id_{\sf A}(t_{\sf A}(a)), id_{\sf H}(t_{\sf A}(a)))$$
$$=f(id_{{\sf A}\bowtie {\sf H}}(t_{\sf A}(a)))=f(id_{{\sf A}\bowtie {\sf H}}(t_{{\sf A}\bowtie {\sf H}}(a,g)))=
(f\circ \Pi^{L}_{{\mathbb K}[{\sf A}\bowtie {\sf H}]})(a,g).
$$

Therefore (\ref{mkl1}) and (\ref{mkl2}) hold. Finally, the morphism $\nabla_{{\mathbb K}[{\sf A}\bowtie {\sf H}]}$ is defined by
$$\nabla_{{\mathbb K}[{\sf A}\bowtie {\sf H}]}((a, g)\otimes (b,h))= \left\{ \begin{array}{l}
(a, g)\otimes (b,h)\;\;\; {\rm if} \;\; (g, b)\in {\sf H}_1\;_{s_{\sf H}}\hspace{-0.15cm}\times_{t_{\sf A}} {\sf A}_1,\\ 
\\
0 \;\; {\rm otherwise,}
\end{array}\right.
$$
and then (\ref{mkl4}) holds because 
\begin{align*}
\;	& \;\;\;\;\;(\mu_{{\mathbb K}[{\sf A}]\bowtie {\mathbb K}[{\sf H}]}\circ (f\otimes f)\circ \nabla_{{\mathbb K}[{\sf A}\bowtie {\sf H}]})((a, g)\otimes (b,h))\\
\;	& = \left\{ \begin{array}{l}
a\bullet \varphi_{\sf A}(g, b)\otimes \phi_{\sf H}(b,h)\star h\;\;\; {\rm if} \;\; (g, b)\in {\sf H}_1\;_{s_{\sf H}}\hspace{-0.15cm}\times_{t_{\sf A}} {\sf A}_1,\\ 
\\
0 \;\; {\rm otherwise}
\end{array}\right.
\\
\;	& = \left\{ \begin{array}{l}
f((a, g)._{\Psi} (b,h))\;\;\; {\rm if} \;\; (g, b)\in {\sf H}_1\;_{s_{\sf H}}\hspace{-0.15cm}\times_{t_{\sf A}} {\sf A}_1,\\ 
\\
0 \;\; {\rm otherwise}
\end{array}\right.
\\
\;	& =(f\circ \mu_{{\mathbb K}[{\sf A}\bowtie {\sf H}]})((a, g)\otimes (b,h)).
\end{align*}
\end{proof}

Taking into account that quasigroups are examples of quasigroupoids, we have that following corollary.
\begin{corollary}
Let $(A, H)$ be a matched pair of  quasigroups. The cocommutative   Hopf quasigroups ${\mathbb K}[{ A}\bowtie {H}]$ and ${\mathbb K}[{ A}]\bowtie {\mathbb K}[{ H}]$ are isomorphic in the category of Hopf quasigroups.
\end{corollary}

\section{Funding} The  author was supported by  Ministerio de Ciencia e Innovaci\'on of Spain. Agencia Estatal de Investigaci\'on. Uni\'on Europea - Fondo Europeo de Desarrollo Regional (FEDER). Grant PID2020-115155GB-I00: Homolog\'{\i}a, homotop\'{\i}a e invariantes categ\'oricos en grupos y \'algebras no asociativas.


\begin{thebibliography}{99}
	
\bibitem{Mar} M. Aguiar, N. Andruskiewitsch,  
Representations of matched pairs of groupoids and applications to weak Hopf algebras (
de la Peña, José A. (ed.) et al., Algebraic structures and their representations. Proceedings of "XV Coloquio Latinoamericano de Álgebra", Cocoyoc, Morelos, México, July 20-26, 2003. Providence, RI: American Mathematical Society (AMS)) {\em Contemporary Mathematics} {\bf 376} (2005) 127-173.

\bibitem{Asian} J. N. Alonso \'Alvarez, J. M. Fern\'andez Vilaboa,  R. Gonz\'alez Rodríguez Weak Hopf quasigroups,  {\em Asian J. of Math.}  {\bf  20}  (2016) 665-694. 

\bibitem{JPAA}   J. N. Alonso \'Alvarez, J. M. Fern\'andez Vilaboa,  R. Gonz\'alez Rodríguez, Cleft and Galois extensions associated to a weak Hopf quasigroup, {\em  J. Pure Appl. Algebra}  {\bf 220} (2016) 1002-1034.

\bibitem{MED} J. N. Alonso \'Alvarez, J. M. Fern\'andez Vilaboa,  R. Gonz\'alez Rodríguez, A characterization of weak Hopf (co)quasigroups, {\em Mediterr. J. Math.}  {\bf 13} (2016) 3747-3764. 

\bibitem{our2} J. N. Alonso \'Alvarez, J. M. Fern\'andez Vilaboa,  R. Gonz\'alez Rodríguez,  Multiplication alteration by two-cocycles. The  nonassociative version, {\em Bull. Malays. Math. Sci. Soc.}  {\bf 43}  (2020) 3557-3615.

\bibitem{JA21} J. N. Alonso \'Alvarez, J. M. Fern\'andez Vilaboa,  R. Gonz\'alez Rodríguez,  Quasigroupoids and weak Hopf quasigroups, {\em J. Algebra}  {\bf 568} (2021) 408-436.

\bibitem{AN} N. Andruskiewitsch, S. Natale, 
Double categories and quantum groupoids, {\em Publ. Mat. Urug.}  {\bf 10} (2005) 11-51. 

\bibitem{GaPe} G. Böhm, J.  Gómez-Torrecillas,  
 On the double crossed product of weak Hopf algebras
(Andruskiewitsch, Nicolás (ed.) et al., Hopf algebras and tensor categories. Proceedings of the international conference, University of Almería, Almería, Spain, July 4?8, 2011. Providence, RI: American Mathematical Society   (AMS)) {\em  Contemporary Mathematics}  {\bf 585} (2013) 153-173.

\bibitem{BRANDT}  H. Brandt,   Uber eine Verallgemeinerung des Gruppenbegriffes, {\em Math. Ann.}  {\bf 96} (1926) 360-366. 

\bibitem{Bruck} R. H. Bruck,  Contributions to the theory of loops, {\em Trans. Amer. Math. Soc.}  {\bf 60} (1946)  245-354.

\bibitem{CHEIN} O. Chein, Moufang loops of small order I, {\em Trans. Amer. Math. Soc.}  {\bf 188} (1974) 31-51. 

\bibitem{GRABO} J. Grabowski, An introduction to loopoids, {\em Comment. Math. Univ. Carolin.}  {\bf 57} (2016) 515-526. 

\bibitem{GRABO22}  J. Grabowski, Z. Ravanpak,  Nonassociative analogs of Lie groupoids, {\em Differential Geom. Appl.}  {\bf 82} (2022)  101887. 

\bibitem{HAHN} P. Hahn, Haar measure for measure groupoids, {\em Trans. Amer. Math. Soc.}  {\bf 242} (1978) 1-33. 

\bibitem{SM2}  B. Im, A.W. Nowak, J. D. H. Smith,  Algebraic properties of quantum quasigroups, {\em J. Pure Appl. Algebra}  {\bf 225} (2021) 106539.


\bibitem{KM}  J. Klim, S. Majid, Hopf quasigroups and the algebraic 7-sphere,  {\em J. Algebra}  {\bf 323} (2010) 3067-3110. 

\bibitem{KM2}  J. Klim, S. Majid, Bicrossproduct Hopf quasigroups, {\em Comment. Math. Univ. Carolin.}  {\bf 51} (2010)  287-304.

 
\bibitem{LW}  J. -H. Lu, A. Weinstein,  Poisson Lie groups, dressing transformations and Bruhat decompositions, {\em J. Differential Geom.}   {\bf 31} (1990) 501-526. 

\bibitem{Mac} K. Mackenzie,  Double Lie algebroids and second-order Geometry, I,  {\em  Adv. Math.}  {\bf  94} (1992) 180-239.

\bibitem{MAJ} S. Majid,  Physics for algebraist: non-commutative and non-cocommutative Hopf algebras by a bicrossproduct construction,  {\em  J. Algebra}   {\bf 130} (1990)  17-64.

\bibitem{MAJDCP} S. Majid, {\em Foundations of Quantum Group Theory} (Cambridge University Press, Cambridge, 1995).


\bibitem{PIS} J. M. P\'erez-Izquierdo, I.P.  Shestakov,  An envelope for Malcev algebras, {\em J. Algebra}  {\bf 272} (2004) 379-393.

\bibitem{PI07} J. M. P\'erez-Izquierdo, Algebras, hyperalgebras, non-associative bialgebras and loops, {\em  Adv. Math.}  {\bf 208} (2007) 834-876. 

\bibitem{SM1} J. D. H. Smith, Quantum quasigroups and loops, {\em J. Algebra}  {\bf 456} (2016) 46-75.

\bibitem{TAK1} M. Takeuchi, Matched pairs of groups and bismash products of Hopf algebras, {\em Comm. Algebra} {\bf 9} (1981) 841-882.


\end{thebibliography}
\end{document}